\documentclass[11pt,reqno]{amsart}
\usepackage[margin=1in]{geometry}
\usepackage{amsmath,amssymb,amsthm,mathtools}
\usepackage{booktabs,longtable,array,multicol}
\usepackage[T1]{fontenc}
\usepackage{makebox}
\usepackage{lmodern,microtype,bm}
\usepackage[normalem]{ulem}
\usepackage{xcolor}
\usepackage[colorlinks=true,linkcolor=blue!45!black,citecolor=blue!45!black,
 urlcolor=blue!45!black,hypertexnames=false]{hyperref}
\usepackage{listings}
\allowdisplaybreaks[2]
\numberwithin{equation}{section}
\theoremstyle{plain}
\newtheorem{theorem}{Theorem}[section]
\newtheorem{lemma}[theorem]{Lemma}
\newtheorem{proposition}[theorem]{Proposition}
\newtheorem{corollary}[theorem]{Corollary}
\theoremstyle{definition}

\theoremstyle{remark}
\newtheorem{remark}[theorem]{Remark}
\newcommand{\C}{\mathbb C}

\newcommand{\E}{\mathbb E}
\newcommand{\re}{\operatorname{Re}}
\newcommand{\im}{\operatorname{Im}}

\newcommand{\dd}{\,\mathrm dt}

\newcommand{\RepositoryURL}{\href{https://github.com/LudovickBouthat/borcea-verifier}{https://github.com/LudovickBouthat/borcea-verifier}}
\newcommand{\ReleaseURL}{\href{https://github.com/LudovickBouthat/borcea-verifier/releases/tag/v1.0.0}{https://github.com/LudovickBouthat/borcea-verifier/releases/tag/v1.0.0}}
\newcommand{\ArchiveDOI}{\href{https://doi.org/10.5281/zenodo.23026060}{doi:10.5281/zenodo.23026060}}

\makeatletter
\newcommand{\itemtag}[1]{%
  \item[#1]%
  \csname phantomsection\endcsname  % makes the hyperref link point to the item (does nothing without hyperref)
  \protected@edef\@currentlabel{#1}%
}
\makeatother

\title{A proof of Borcea's quadratic variance conjecture}

\author{Ludovick Bouthat}
\address{Department of Mathematical Sciences, University of Delaware, Newark, DE, USA}
\email{lbouthat@udel.edu}

\author{Shaun Fallat}
\address{Department of Mathematics and Statistics, University of Regina, Regina, SK, Canada}
\email{shaun.fallat@uregina.ca}

\date{}

\begin{document}
\begin{abstract}
Let $p$ be a complex polynomial of degree $n\ge2$, and let $\sigma_2(p)$ be the root mean square distance of its zeros from their centroid, including multiplicities. We prove that every zero of $p$ lies within distance $\sigma_2(p)$ of a critical point. This proves Borcea's quadratic conjecture and, in particular, gives a new proof of Sendov's conjecture. Two polynomial integration identities prove the result analytically for $n\ge2048$. For the remaining degrees, we derive explicit necessary inequalities, including a variable integration endpoint and an exact centered coefficient identity. Directed integer arithmetic certifies a finite covering of the resulting parameter domains. Degrees $4$--$26$ use $150\,845$ certified three-dimensional regions; degrees $27$--$2047$ use a further $27\,429$ regions, with degree intervals allowed. The cases of degree 2 and 3 were already covered by Khavinson--Pereira--Putinar--Saff--Shimorin.
\end{abstract}
\maketitle

\section{Introduction}\label{sec:introduction}

In 1958, Blagovest Sendov asked whether every zero of a complex polynomial whose zeros all lie in the closed unit disk has a critical point within distance one. The problem later appeared in Hayman's 1967 collection~\cite{Hay67} under the name Ilieff's conjecture; Sendov's own account explains both its origin and the somewhat complicated circumstances surrounding that attribution~\cite[pp.\;284--285]{Sen02}. The question was motivated by the classical Gauss–Lucas theorem, which asserts that all critical points of a polynomial lie in the convex hull of its zeros. When the zeros lie in a disk of radius $r$, Gauss–Lucas yields a distance bound of $2r$ from a prescribed zero to any critical point. Sendov asked if this factor of two can be reduced, which naturally led to the question of whether one can always find at least one critical point within distance $r$ of every zero~\cite[p.~284]{Sen02}. Thus, the conjecture seeks local information about the location of critical points that is not provided by the convex-hull theorem. The polynomial $p(z)=z^n-1$ shows that the radius-one bound is the best possible and cannot be reduced.

Early progress focused on low-degree polynomials and special configurations of zeros. The case $n=2$ is trivial, Brannan~\cite{Bra68} and Rubinstein~\cite{Rub68} proved it for $n=3$, Rubinstein also treated $n=4$ (see also~\cite{Joy69,Sch69}), and Meir and Sharma extended the result to all $n\le 5$ in 1969~\cite{MS69}. Borcea established the conjecture for degrees six and seven in 1996~\cite{Bor96a,Bor96b} (see also~\cite{Bro91,Bro97}), and Brown and Xiang extended the result to degree eight in 1999~\cite{BX99}. In parallel, the conjecture was verified for several classes of polynomials defined by the geometry of their zeros: when the distinguished zero lies on the unit circle~\cite{GRR69,Joy69,Rub68,Sch69}; when one of the zeros is at the origin~\cite{BRS85,Sch69}; when all the zeros are real~\cite{RS02}; when the convex hull of the zeros is a triangle~\cite{Sch77}; for trinomials~\cite{Sch71}; and for polynomials with at most three or four distinct zeros~\cite{CS88a,CS88b}. Bojanov, Rahman, and Szynal also showed that, in every degree, the conclusion holds if the radius $1$ is replaced by $(1+|z_1\cdots z_n|)^{1/n}\le 2^{1/n}$, where $z_1,\dots,z_n$ are the zeros of the polynomial~\cite{BRS85}. We refer to Sendov~\cite{Sen02} and Khavinson--Pereira--Putinar--Saff--Shimorin~\cite{KPPSS} for more detailed accounts of these developments.

A major breakthrough occurred when Tao proved the conjecture for all sufficiently large degrees, first announced in 2020 and then published in 2022~\cite[Theorem~1.2]{Tao22}. This reduced the problem to a finite, although unspecified, set of remaining degrees. An explicit threshold was obtained only recently by Zhang, who showed that the conjecture holds for every degree $n\ge 10^{200000}$~\cite{Zha26}. In August 2026, Mazur completed the proof with a computer-assisted argument that covered all degrees~\cite{Maz26}. Tao's subsequent exposition distilled the proof into a small collection of polynomial identities that provide an intuitive organizational framework for the argument~\cite[Sections~1--4]{Tao26}.

These advances also stimulated interest in stronger variants of the conjecture. In particular, Tang and Zhang proposed inequalities involving the averages of reciprocal critical-point distances~\cite[Conjecture~1.10]{TZ25}. In September 2026, Zhang proved the quadratic member of that family~\cite{Zhang26}: if the zeros of $p$ lie in the closed unit disk and
$\zeta_1,\ldots,\zeta_{n-1}$ are its critical points, then at every zero $a$,
\begin{equation*}\label{eq:TZ}
 \sum_{j=1}^{n-1}|a-\zeta_j|^{-2}\ge n-1.
\end{equation*}
A zero denominator is interpreted as $+\infty$. Note that Zhang's result controls an average of distances, whereas Sendov's assertion controls their minimum. Zhang's argument combines a centered product expansion, estimates for large degrees, and exact verification of the remaining parameter ranges~\cite[Sections~4--7]{Zhang26}, techniques which are going to be relevant to this paper.

Borcea proposed a different strengthening, replacing the radius of a containing disk by a measure of the average spread of the zeros. For $p(z)=c\prod_{i=1}^n(z-z_i)$, $c\ne0$, and $1\le q<\infty$, set
\[
 \sigma_q(p)=\min_{b\in\C} \left(\frac1n\sum_{i=1}^n|z_i-b|^q\right)^{\!1/q},
\]
where zeros and critical points are counted with multiplicity. Borcea's variance conjectures assert that
\begin{equation}\label{eq:borcea}
 \max_{p(a)=0}\min_{p'(\zeta)=0}|a-\zeta|\le\sigma_q(p).
\end{equation}
The limiting quantity $\sigma_\infty(p)=\min_b\max_i|z_i-b|$ is the radius of the smallest containing disk, so its version of \eqref{eq:borcea} is precisely equivalent to Sendov's conjecture. For finite $q$ the proposed bound is stronger, and it takes root multiplicities into account. Khavinson, Pereira, Putinar, Saff, and Shimorin developed this family of questions, proved several special cases, and related them to normal matrix compressions and Cauchy transforms~\cite{KPPSS}. In particular, they established the quadratic case for polynomials with at most three distinct zeros and the $q=1$ case for cubics~\cite[Sections~3--4]{KPPSS}.

In discussing further directions on his blog, Tao observed that the polynomial identities used for Sendov remain valid under a moment constraint, whereas the inequalities based on confinement to a disk do not carry over. He also noted some preliminary explorations of related questions that did not lead to a resolution of the conjecture~\cite[Section~5]{Tao26}. The obstruction to proving Borcea's conjecture thus appears to lie in estimating the identities, rather than in establishing the identities themselves. In this paper, we address it for $q=2$ and completely prove Borcea's quadratic variance conjecture. 

\begin{theorem}\label{thm:main}
Let $p\in\C[z]$ have degree $n\ge2$. For every zero $a$ of $p$, there is a zero $\zeta$ of $p'$ such that
\begin{equation*}\label{eq:main}
 |a-\zeta|\le\sigma_2(p).
\end{equation*}
\end{theorem}

The constant is sharp, again by $p(z)=z^n-1$. The proof is analytic for $n\ge2048$, and uses a finite certificate verified by directed integer arithmetic for degrees $4\le n\le2047$. The cases $n=2$ and $n=3$ were already established in \cite{KPPSS}.

\subsection{The obstruction and the proof strategy}
Our argument draws inspiration from the integral identities in the recent work on Sendov~\cite{Maz26,Tao26} and Zhang's conjecture~\cite{Zhang26}. Their use here requires an additional estimate adapted to variance normalization. Indeed, after shifting and scaling to impose $\sum_{i=1}^n z_i=0$ and $\sigma_2(p)=1$, the condition on the original zeros is $\sum_i|z_i|^2=n$; individual zeros can lie outside the unit disk. Therefore, the elementary disk comparison
\[
 |1-az|^2-|a-z|^2=(1-a^2)(1-|z|^2)
\]
that was crucially used in previous work has no prescribed sign and cannot be used.

The proof of the main result is achieved by contradiction. Following the reciprocal coordinates used in~\cite{Tao26,Zhang26}, write
\[
 q_j := (a-\zeta_j)^{-1},\qquad
 \mu := \frac1{n-1}\sum_jq_j,\qquad
 s := \frac1{n-1}\sum_j|q_j|^2.
\]
The assumed counterexample is equivalent to having $|q_j|<1$ for every $j$. The key analog for the disk geometry inequality above is the weighted Cauchy--Schwarz estimate proved in Lemma~\ref{lem:deficit}:
\begin{equation*}\label{eq:introdeficit}
 |a-\overline\mu|^2 \le (1-s)\bigg(a^2-\frac1{n-1}\bigg).
\end{equation*} 
This shows that when either $a$ is close to $(n-1)^{-1/2}$ or when the average of the $1-|q_j|^2$ is small, then the mean of the $q_j$ must be close to $a$. These two restrictions replace the point-wise information about the original zeros that the disk argument used. As such, several inequalities in this paper are treated in two cases to obtain sharp bounds in both of the regimes and obtain a contradiction.

\subsection{Organization and computational scope}

Section~\ref{sec:normalization} sets up a hypothetical counterexample and derives some moment and product restrictions that the points $q_j$ must satisfy. Section~\ref{sec:integrals} proves two integration estimates, and Section~\ref{sec:large} combines them to obtain a contradiction in every degree $n\ge 2048$. In most cases, a direct estimate is enough: the error term in a product identity satisfied by the $q_j$ is too small to cancel its constant term. This direct estimate fails only in a narrow range of parameters. In that range, we compare the product with the simpler product obtained by replacing every $q_j$ by their average $\mu$. The error from this replacement is proportional to the variance $s-|\mu|^2$ of the $q_j$. It is thus small when the points are close together; when they are spread out, the larger error is offset by a correspondingly smaller value of $|\mu|$. Either way, we again reach a contradiction.

For the remaining degrees $4\le n\le 2047$, these estimates are not precise enough. Instead, we improve the bounds of Section~\ref{sec:integrals}, at the cost of them becoming more complicated. Section~\ref{sec:finitebounds} first outlines more precisely this finite-degree argument. It then reduces the hypothetical counterexample to a point in a three-dimensional search space, records the bounds that the coordinates of this point must satisfy, and derives three explicit exclusion tests. Subsection~\ref{sec:boundary} adds the one analytic ingredient that falls outside this scheme, a local estimate in degree four. Section~\ref{sec:computer} covers the search space by boxes, checks in certified arithmetic that every box is excluded, and describes the resulting verified certificates, which complete the proof for every degree $n\ge 2$. The appendix identifies the implementation and gives instructions for reproducing the computation.

\section{Normalization and reciprocal moments}\label{sec:normalization}

We begin the contradiction argument by proving the restrictions that every counterexample must satisfy. The normalization \eqref{eq:normal} fixes the scale of the original roots and the reciprocal coordinates put the critical points in the open unit disk. The purpose of this section is to relate those two descriptions. The normalization assumptions introduced below hold throughout Sections~\ref{sec:normalization}--\ref{sec:computer}.

If $\sigma_2(p)=0$, all zeros coincide and the theorem is immediate since a prescribed multiple zero is itself a critical point. We may therefore assume that $\sigma_2(p)>0$ and that the prescribed zero is simple. Translation by $\frac{1}{n}\sum_i z_i$, division of the variable by $\sigma_2(p)$, rotation, and multiplication of the polynomial by a nonzero constant allow us to impose, without loss of generality, that $\sigma_2(p)=1$ and
\begin{equation}\label{eq:normal}
 p(z)=\prod_{i=1}^n(z-z_i),\qquad
 \sum_i z_i=0,\qquad \sum_i|z_i|^2=n,\qquad z_1=a\ge0.
\end{equation}
We write $m:=n-1$ for the number of critical points, with multiplicity. Under these assumptions, failure of Theorem~\ref{thm:main} at the prescribed zero now means
\begin{equation}\label{eq:counter}
 |a-\zeta_j|>1\qquad(1\le j\le m).
\end{equation}

Our first lemma transfers the centroid and second-moment information from the roots described above to the critical points. The second-moment inequality uses Schoenberg's inequality, proposed by Schoenberg \cite{Schoenberg86} and proved independently by Pereira \cite{Pereira03} and Malamud \cite{Malamud05} in the early 2000s.

\begin{lemma}\label{lem:critical}
Under \eqref{eq:normal}, the critical points satisfy
\begin{equation*}\label{eq:critical}
 \sum_{j=1}^m\zeta_j=0,\qquad
 \kappa:=\frac1m\sum_{j=1}^m|\zeta_j|^2\le\frac{m-1}{m}.
\end{equation*}
Moreover, $0\le a^2\le m$.
\end{lemma}
\begin{proof}
By \eqref{eq:normal}, the coefficient of $z^{n-1}$ in $p$ vanishes. Consequently, the coefficient of $z^{m-1}$ in $p'/n$ also vanishes, which proves the centroid assertion by Vieta's formula. Now, by Schoenberg's inequality and \eqref{eq:normal},
\[
\sum_{j=1}^m|\zeta_j|^2
\le \frac{n-2}{n}\sum_{i=1}^n|z_i|^2
=n-2=m-1,
\]
and therefore $\kappa\le\frac{m-1}{m}$. Finally, since \(z_1=a\) and $\sum_i z_i=0$, we have $\sum_{i=2}^n z_i=-a$. Consequently,
\[
0 \le \sum_{i=2}^n\left|z_i+\frac{a}{m}\right|^2 = n-a^2-\frac{a^2}{m}.
\]
Since \(n=m+1\), this is equivalent to \(a^2\le m\). The lower bound \(a^2\ge0\) is immediate.
\end{proof}

If $a=0$, then \eqref{eq:counter} would give $\kappa>1$, contradicting Lemma~\ref{lem:critical}. Hence, we need to have $a>0$. As we did earlier, define the reciprocal points, their mean, their second moment, and their variance by
\begin{equation}\label{eq:moments}
 q_j:=\frac1{a-\zeta_j},\qquad
 \mu:=\E(q),\qquad s:=\E\bigl(|q|^2\bigr),\qquad
 v:=\E\bigl(|q-\mu|^2\bigr) = s-|\mu|^2,
\end{equation}
where $\E$ denotes the uniform average over $1\le j\le m$. Thus, $0<|q_j|<1$, $0<s<1$, and $0\le v\le s$.

Our next lemma essentially shows that $\mu$ must converge toward $a$ as $s$ approaches one, and that $1-a\mu$ must be small when the reciprocal points cluster. Both of these estimations will play an important role in controlling the products in Section~\ref{sec:integrals}.

\begin{lemma}\label{lem:deficit}
Under \eqref{eq:normal} and \eqref{eq:counter}, we have
\vspace{-5pt}
\begin{multicols}{3}
\begin{enumerate}
    \itemsep=5pt
    
    \itemtag{(1)} $\E\bigl(q^{-1}\bigr) = a$;\label{lem:deficit1}

    \itemtag{(4)} $|a-\overline\mu|^2 \hspace{-1pt}\le\hspace{-1pt} (1 \hspace{-1pt}-\hspace{-1pt} s)(a^2 \hspace{-1pt}-\hspace{-1pt} \frac{1}{m})$;\label{lem:deficit8}

    \itemtag{(2)} $\E\bigl(|q|^{-2}\bigr) = a^2+\kappa$;\label{lem:deficit2}

    \itemtag{(5)} $\re(\mu)\ge a(1-\sqrt{1-s})$;\label{lem:deficit5}

    \itemtag{(3)} $|1-a\mu|^2\le\kappa v\le\frac{m-1}{m}v$;\label{lem:deficit7}
    
    \itemtag{(6)} $|\mu|\le a(1+\sqrt{1-s})$.\label{lem:deficit6}

\end{enumerate}
\end{multicols}
\vspace{-7pt}
\end{lemma}
\begin{proof}
The identities $q_j^{-1}=a-\zeta_j$ and $\E(\zeta)=0$ from Lemma~\ref{lem:critical} give the identities \ref{lem:deficit1} and \ref{lem:deficit2}. Now, note that we have $\zeta_jq_j=aq_j-1$. Hence, averaging and using $\E(\zeta)=0$ gives
\[
 a\mu-1=\E(\zeta q)=\E\bigl(\zeta(q-\mu)\bigr).
\]
Cauchy--Schwarz and \eqref{eq:moments} therefore give $|a\mu-1|^2 \le \E\bigl(|\zeta|^2\bigr)\E\bigl(|q-\mu|^2\bigr) = \kappa v$, proving \ref{lem:deficit7}.

Define $\delta_j := 1-|q_j|^2 > 0$. The identity $q_j^{-1}-\overline{q_j}=\delta_j/q_j$ and the first inverse moment give $a-\overline\mu = \E(\delta/q)$. We may thus apply Cauchy--Schwarz with the two factors $\sqrt\delta$ and $\sqrt\delta/q$ to obtain
\[
 |a-\overline\mu|^2 \le \E(\delta)\,\E\biggl(\frac{\delta}{|q|^2}\biggr)
 =(1-s)(a^2+\kappa-1).
\]
Applying Lemma~\ref{lem:critical} then gives \ref{lem:deficit8}. Moreover, it follows that $|a-\overline\mu|\le a\sqrt{1-s}$. Taking real parts on both sides yields \ref{lem:deficit5}, while the triangle inequality $|\mu|\le a+|\overline\mu-a|$ gives \ref{lem:deficit6}.
\end{proof}

We shall also need information involving all roots simultaneously. The following lemma estimates $p(0)$ and $p'(a)$ in two different ways. In particular, they prevent the distinguished zero $a$ from being too small: the lower bound for $a^2$ in \eqref{eq:productnecessary} will set the scale of 2048 for the large-degree proof.

\begin{lemma}\label{lem:products}
Let $\Pi:=\prod_{j=1}^m|q_j|$. Suppose that $n\geq 4$. Then
\begin{align}
 &\Pi \le s^{m/2},\qquad
 |p(0)| \le a\bigg(\frac{n-a^2}{m}\bigg)^{\!m/2}
                \!\le\min(1,\sqrt{e}\,a),\nonumber\\
 &1\le\frac{m}{n}n^{2/m} \le (1+a^{2})s\le2(1+\sqrt{1-s}).\label{eq:productnecessary}
\end{align}
In particular, $a^2>(2\log n-1)/n$. Here $e$ is the base of the natural logarithm.
\end{lemma}
\begin{proof}
AM--GM applied to the $m$ numbers $|q_j|^2$ gives $\Pi^{2/m}\le s$. Applied to $|z_2|^2,\ldots,|z_n|^2$, whose sum is $n-a^2$ by \eqref{eq:normal}, we obtain the stated bound on $|p(0)|$. To see that this explicit bound is at most one, apply AM--GM to the $n$ numbers $a^2,(n-a^2)/m,\ldots,(n-a^2)/m$, whose sum is $n$. It is also at most $\sqrt{e}\,a$, since $((n-a^2)/m)^m\le (n/m)^m = (1+1/m)^m<e$.

The derivative has two useful factorizations at $a$:
\[
 |p'(a)|=\prod_{i=2}^n|a-z_i|=\frac n\Pi.
\]
The centroid and second-moment conditions in \eqref{eq:normal} give
\[
 \sum_{i=2}^n|a-z_i|^2
 =ma^2-2a\sum_{i=2}^n \re(z_i)+\sum_{i=2}^n|z_i|^2
 =n(1+a^2).
\]
Consequently AM--GM implies $n/\Pi = \prod_{i=2}^n|a-z_i| \le (n(1+a^2)/m)^{m/2}$. Combining this with $\Pi\le s^{m/2}$ and the fact that $\frac mnn^{2/m} \geq 1$ for all $n\geq 4$ yields the first two inequalities in \eqref{eq:productnecessary}. Moreover, an application of the Cauchy--Schwarz inequality gives 
\begin{equation}\label{eq:CS}
    |\mu| = |\E(q)| \le \sqrt{\E(|q|^2)} = \sqrt{s}.
\end{equation} 
Together with the bound $ah\le\re(\mu)\le|\mu|$, where $h:=1-\sqrt{1-s}$, from \ref{lem:deficit5} in Lemma~\ref{lem:deficit}, this implies $a^2s\le(s/h)^2$. Here $h>0$ because $s>0$ and $s/h=1+\sqrt{1-s}$. Hence, adding $s$ on both sides and simplifying proves the desired upper bound. Lastly, since $\Pi<1$ by \eqref{eq:counter}, the same calculation gives
\[
 a^2>\frac mn\exp\!\left(\frac{2\log n}{m}\right)-1
 \ge\frac{2\log n-1}{n},
\]
where the last step uses $e^x\ge1+x$ and $m=n-1$.
\end{proof}

\section{Polynomial integration identities}\label{sec:integrals}

The restrictions of Section~\ref{sec:normalization} concern moments and products separately. We now connect them by interpolating and integrating $p'$ along the segment from the prescribed zero $a$ to the centroid $0$. The estimate will suffice in two of the three ranges treated in Section~\ref{sec:large}. A third centered version will cover the remaining range by making the error proportional to the reciprocal variance $v$.

\subsection{The product along the segment}
For $0\le t\le1$, define
\begin{equation*}\label{eq:FDbeta}
 F(t):=\prod_j(1-atq_j),\qquad
 \beta(t):=\E|1-atq|^2=1-2a\re(\mu)\,t+a^2st^2.
\end{equation*}
A direct computation shows that $F(t)=\smash{p'(a(1-t))/p'(a)}$. Thus $F$ encodes the derivative on the segment, while $\beta$ is the mean square of its individual factors. Bounds on $\beta$ therefore translate into bounds on $F$ by the AM--GM inequality. For clarity and conciseness of what follows, we once again use $h=1-\sqrt{1-s}$. 
The following consequences of Lemma~\ref{lem:deficit} provide two complementary decreasing rate estimates for $\beta$ along the segment between $0$ and the root $a$.

\begin{lemma}\label{lem:beta}
Under \eqref{eq:normal} and \eqref{eq:counter}, we have
\begin{align}
 &0\le\beta(t)\le1-a^2h^2t\le1,\qquad \beta(t)\le1-a^2s\,t(1-t).\label{eq:betachord}
\end{align}
\end{lemma}
\begin{proof}
Lemma~\ref{lem:deficit}, in particular the inequality \ref{lem:deficit5}, gives $\re(\mu) \ge ah$. Since $2h-s=h^2$, it follows that $\beta(1)\le1-2a^2h+a^2s=1-a^2h^2$. The function $\beta$ is nonnegative by definition and convex because $\beta''=2a^2s\ge0$. The inequality $\beta(t) \leq (1-t)\beta(0)+t\beta(1)$ therefore gives the first inequality. Moreover, it is easily shown that $2h = 2(1-\sqrt{1-s})\ge s$. Thus, another use of inequality \ref{lem:deficit5} in Lemma~\ref{lem:deficit} yields
\[
 \beta(t)\le1-2a^2ht+a^2st^2 \le 1-a^2st+a^2st^2,
\]
which is the second inequality that we wanted to prove.
\end{proof}

\subsection{The direct estimate}
The next identity compares the value $F(0)=1$ with three quantities: an endpoint term $|F(1)|$, a product term controlled by Lemma~\ref{lem:products}, and an integral remainder. The remainder has two complementary bounds to treat the two regimes discussed in the introduction. In particular, one estimation is precise near $t=0$, while the other controls it near $t=1$. In the following, we shall use the quantities
\[
 c_m:=\left(\frac m{m-1}\right)^{\!\frac{m-1}{2}}\!,\qquad
 \gamma(t):=s\bigl(a^2(1-t)^2+\kappa\bigr),\qquad R:=a^2\!\int_0^1t\sum_jq_j^2 \prod_{\ell\ne j}(1-atq_\ell)\dd.
\]

\begin{lemma}\label{lem:direct}
Under \eqref{eq:normal} and \eqref{eq:counter}, we have
\begin{equation}\label{eq:direct}
 1\le |F(1)|+\frac mn|\mu|\Pi|p(0)|+|R|.
\end{equation}
Moreover, the relevant quantities satisfy the estimates
\begin{align}
 |R|&\le c_mma^2s\int_0^1t \min\{\beta(t),\gamma(t)\}^{\frac{m-1}{2}}\dd, \label{eq:remainder}\\
 |F(1)|&\le\min\bigl\{\beta(1),\kappa s\bigr\}^{m/2}. \label{eq:Fend}
\end{align}
\end{lemma}
\begin{proof}
Recall that $F(t)=\smash{p'(a(1-t))/p'(a)}$. Hence, since $p(a)=0$, integration on $[0,1]$ gives
\begin{equation}\label{eq:origin}
 p(0) = -ap'(a)\int_0^1 F(t)\dd.
\end{equation}
A direct computation shows that the derivative of $F$ is
\(
F'(t) = -a\sum_{j} q_j \prod_{\ell\neq j}(1-atq_\ell).
\)
For each $j$, write $q_j=q_j(1-atq_j)+atq_j^2$ inside the derivative of $F$. After summing, this gives the identity
\[
 F'(t) = -ma\mu F(t)-a^2t\sum_jq_j^2 \prod_{\ell\ne j}(1-atq_\ell).
\]
Integrating and using $F(0)=1$ yields $1 = F(1)+ma\mu\int_0^1F\dd+R$. Hence, Equation~\eqref{eq:origin} and the identity $|p'(a)|=n/\Pi$ from Lemma~\ref{lem:products} prove \eqref{eq:direct}.

For the first remainder bound, apply the AM--GM inequality to obtain
\[
 \prod_{\ell\ne j}|1-atq_\ell| \le \Bigg(\frac{\sum_{\ell\ne j}|1-atq_\ell|^2}{m-1}\Bigg)^{\!\!\frac{m-1}{2}} \!\!\le \Bigg( \frac{m}{m-1}\cdot\frac{\sum_{\ell}|1-atq_\ell|^2}{m}\Bigg)^{\!\!\frac{m-1}{2}} \!\!= c_m\beta(t)^{\frac{m-1}{2}}.
\]
Summing with weights $|q_j|^2$ gives a bound
$\smash{c_mms\beta(t)^{\frac{m-1}{2}}}$ for the absolute sum in $R$. For the second part of \eqref{eq:remainder}, factor the same summands as
\[
 q_j^2\prod_{\ell\ne j}(1-atq_\ell) = q_j \cdot \prod_\ell q_\ell \cdot \prod_{\ell\ne j}(a(1-t)-\zeta_\ell).
\]
The centroid identity of Lemma~\ref{lem:critical} gives, after a direct computation, the identity
\[
 \sum_{\ell=1}^m|a(1-t)-\zeta_\ell|^2 = m\bigl(a^2(1-t)^2+\kappa\bigr).
\]
Hence, applying AM--GM and then the bound $\Pi\le s^{m/2}$ from Lemma~\ref{lem:products}, we obtain
\begin{align*}
    |q_j|^2\prod_{\ell\ne j}|1-atq_\ell| \leq |q_j| \Pi\bigg(\frac{m}{m-1}\cdot \frac{\sum_{\ell}|a(1-t)-\zeta_\ell|^2}{m}\bigg)^{\!\!\frac{m-1}{2}} \!\!\leq c_m|q_j|\sqrt{s}\,\gamma(t)^{\frac{m-1}{2}}.
\end{align*}
Since $\sum_j|q_j|\le m\sqrt{s}$ by Cauchy--Schwarz, the absolute sum in $R$ is also bounded by $c_mms\gamma(t)^{(m-1)/2}$. Both estimates hold for every $t$, so taking their minimum before integration proves \eqref{eq:remainder}.

Lastly, AM--GM gives $|F(1)|\le\beta(1)^{m/2}$. Alternatively, we have $1-aq_j=-\zeta_jq_j$ which yields
\[
 |F(1)|=\Pi\prod_j|\zeta_j|
 \le s^{m/2}\kappa^{m/2},
\]
where we used Lemma~\ref{lem:products} and one more application of AM--GM.
\end{proof}

\subsection{The centered estimate}
The direct remainder established in the previous subsection does not capture how close the reciprocal points are to one another. To retain that information, define
\begin{equation}\label{eq:defw}
    D(t) := 1-at\mu,\qquad \omega_j:=a(\mu-q_j),
\end{equation}
so that $1-atq_j=D(t)+t\omega_j$. Then $D^m(t)$ is the product obtained from $F(t)$ by replacing each $q_j$ with its mean. The first-order error vanishes because $\sum_j\omega_j=0$, resulting in a second order estimate in terms of $v$. This is the content of the following lemma. To better state the result, we define the quantities
\begin{equation}\label{eq:defEC}
    E_c:=na|\mu|\int_0^1|F(t)-D^m(t)|\dd, \qquad c_{m,2}:=\left(\frac m{m-2}\right)^{\!\frac{m-2}{2}}.
\end{equation}

\begin{lemma}\label{lem:centered}
Under \eqref{eq:normal} and \eqref{eq:counter}, we have
\begin{equation}\label{eq:centered}
 1\le |\mu|\Pi|p(0)|+|D(1)|^n+E_c.
\end{equation}
Moreover, the integrand in the definition of $E_c$ satisfies
\begin{equation}
    |F(t)-D^m(t)| \leq \frac{ma^2t^2v}{2}\beta(t)^{\frac{m-2}{2}} \min\!\bigg\{\frac{c_m \sqrt{\beta(t)}}{|D(t)|},\, c_{m,2}(m-1) \bigg\},\label{eq:interp1}
\end{equation}
the first entry being omitted if $D(t)=0$.
\end{lemma}
\begin{proof}
Integrating the derivative of $D^n(t)$ gives $1-D^n(1)=na\mu\int_0^1 D^m(t)\dd$. Moreover, Equation~\eqref{eq:origin} and the identity $p'(a)=n/\prod_jq_j$ give
\[
p(0) \prod_jq_j= -an\int_0^1 F(t) \dd.
\]
Multiplying the latter by $\mu$ and adding it to the former identity thus gives 
\[
 1+p(0)\mu\prod_jq_j-D^n(1)
       =na\mu\int_0^1\bigl(D^m(t)-F(t)\bigr)\dd.
\]
Taking absolute values and using the triangle inequality proves \eqref{eq:centered}.

Now, interpolate between $F$ and $D$ by writing $G_\theta(t)=\prod_j(D(t)+t\theta\omega_j)$, so that $G_0=D^m$ and $G_1=F$. Moreover, note that $\sum_j\omega_j=0$ and $\sum_j|\omega_j|^2=ma^2v$ by \eqref{eq:moments}. In particular,
\begin{equation}\label{eq:interpmoment}
\E\bigl(|D(t)+t\theta\omega|^2\bigr) = |D(t)|^2+a^2t^2\theta^2v \le |D(t)|^2+a^2t^2v = \beta(t).
\end{equation}
Differentiating $G_\theta$ relative to $\theta$ and multiplying by $D(t)$ yields
\[
D(t) \partial_\theta G_\theta(t) = t\sum_j \omega_jD(t)\prod_{\ell\neq j} (D(t)+t\theta\omega_\ell).
\]
Replacing $D(t)$ by
$(D(t)+t\theta\omega_j)-t\theta\omega_j$ in the $j$th term, and using $\sum_j\omega_j=0$ then gives
\[
D(t)\partial_\theta G_\theta(t) = -t^2\theta\sum_j\omega_j^2 \prod_{\ell\ne j}(D(t)+t\theta\omega_\ell).
\]
The same omitted-factor AM--GM estimate used in Lemma~\ref{lem:direct}, now with \eqref{eq:interpmoment}, bounds each product by $c_m\beta^{(m-1)/2}$. If $D(t)\ne0$, integration in $\theta$ therefore proves the first term in \eqref{eq:interp1}:
\[
|F(t)-D^m(t)|\le\frac{t^2}{|D(t)|} \bigg(\sum_j|\omega_j|^2\bigg)c_m\beta(t)^{\frac{m-1}{2}} \int_0^1\theta\,\mathrm{d}\theta = \frac{c_mma^2t^2v}{2|D(t)|}\beta(t)^{\frac{m-1}{2}}.
\]

For an estimate valid even when $D=0$, note that $\partial_\theta G_0(t) = tD^{m-1}(t)\sum_j\omega_j=0$ and apply Taylor's formula with integral remainder
\[
F(t)-D^m(t) = G_1(t)-G_0(t)-\partial_\theta G_0(t) = \int_0^1 (1-\theta)\partial_\theta^2G_\theta(t)\,\mathrm{d}\theta.
\]
The second derivative is $t^2\sum_{j\ne\ell}\omega_j\omega_\ell \prod_{k\ne j,\ell}(D+t\theta\omega_k)$. Omitting two factors in AM--GM bounds each product by $c_{m,2}\beta^{(m-2)/2}$, while Cauchy--Schwarz gives
\[
\sum_{j\ne\ell}|\omega_j\omega_\ell|
=\bigg(\sum_j|\omega_j|\bigg)^{\!2} \!-\sum_j|\omega_j|^2
\le(m-1)\sum_j|\omega_j|^2=m(m-1)a^2v.
\]
Since $\int_0^1(1-\theta)\,\mathrm{d}\theta=1/2$, this proves the second estimate in \eqref{eq:interp1}.
\end{proof}

\section{The analytic proof for degrees at least 2048}\label{sec:large}

We now combine the necessary conditions to rule out a counterexample in every sufficiently large degree. There are three regimes. If the reciprocal second moment $s$ is bounded away from one, the two bounds in Lemma~\ref{lem:direct} control the two opposite ends of the integral. If the prescribed zero is small, the same lemma makes its reciprocal mean $\mu$ small as well, and the direct identity again suffices. In the remaining case we use Lemma~\ref{lem:centered}: its error is proportional to $v$, the variance of the $q_j$. As we shall see, this error can be absorbed by a simultaneous reduction in the mean $\mu$ to ensure that this term is controlled as well.

\begin{theorem}\label{thm:large}
Theorem~\ref{thm:main} holds for every $n\ge2048$.
\end{theorem}
\begin{proof}
Assume a counterexample under the hypothesis \eqref{eq:normal} and \eqref{eq:counter} and define $M:=ma^2$, the degree-adjusted size of the prescribed zero. The lower bound on $a^2$ in Lemma~\ref{lem:products} gives
\begin{equation}\label{eq:M14}
 M > \frac{n-1}{n}\bigl(2\log (n)-1\bigr) > 14.
\end{equation}
For the last comparison, note that $\log(2) = 2\sum_{j\;\text{odd}}\frac{1}{j\cdot3^j} > 2(\frac{1}{3}+\frac{1}{81}) = \frac{56}{81}$. The derivative of $(1-1/n)(2\log(n)-1)$ is $(2n+2\log(n)-3)/n^2>0$ for $n\ge2048$, and at $n=2048$ the expression exceeds $(2047/2048)(22\cdot56/81-1)>14$. We now divide into the three ranges described above.

\subsection*{Case 1: \texorpdfstring{$\bm{s\le3/4}$}{s ≤ ¾}}
In this range, \eqref{eq:productnecessary} gives $a^2s \ge 1-s \ge 1/4$. For $0\le t\le3/4$, the second bound of Lemma~\ref{lem:beta} gives $\beta(t)\le1-a^2st/4\le e^{-a^2st/4}$. Moreover, we have $c_m = (1+\frac{1}{m-1})^{(m-1)/2}<\sqrt{e}<5/3$, which follows from
$(1+1/x)^{x}<e$ and the elementary rational estimate
\[
 e = \sum_{j=0}^\infty\frac1{j!} \le \sum_{j=0}^4\frac1{j!}+\frac{1}{5!}\sum_{j=0}^\infty\frac{1}{6^j} = \frac{1631}{600}<\frac{25}{9}.
\]
Here the tail was bounded by a geometric series using successive ratios at most $1/6$. The contribution on the interval $[0,3/4]$ to the bound
for $|R|$ in Lemma~\ref{lem:direct} is therefore at most
\begin{align}
 c_mma^2s\int_0^\infty t e^{-(m-1)a^2st/8}\dd &= \frac{64c_mm}{(m-1)^2a^2s} < \frac{1280m}{3(m-1)^2} < \frac{432}{m}<\frac14.\label{eq:case1early}
\end{align}

For $3/4\le t\le1$, we instead work with the $\gamma$ bound. Lemma~\ref{lem:critical} gives $\kappa<1$, and \eqref{eq:productnecessary} gives
\[
 \gamma(t)\le s+\frac{a^2s}{16} = \frac{15s}{16} + \frac{(1+a^2)s}{16}
 \le \frac{15s}{16} +\frac{2(1+\sqrt{1-s})}{16}\le\frac{57}{64}.
\]
The last function has derivative $\frac{15}{16}-\frac{1}{16\sqrt{1-s}}>0$ on $[0,3/4]$, so its maximum is $57/64$ at $s=3/4$. Using $t\le1$, $a^2s\le (1+a^2)s \leq 2+2\sqrt{1-s}\leq 4$, and the interval length $1/4$, the remaining contribution is thus at most
\begin{equation}\label{eq:case1tail}
 \frac53m\!\left(\frac{57}{64}\right)^{\!(m-1)/2} \!\le 2n e^{-7(n-2)/128} < \frac1{16}.
\end{equation}
To verify the last inequality for all $n\ge2048$, note that the logarithmic derivative of $n e^{-7(n-2)/128}$ is $1/n-7/128<0$. Also, at $n=2048$, its exponent satisfies $7\cdot2046/128>17$. Hence, using $e>2$, we find $e^{7(2046)/128} > 2^{17} > 32\cdot2048$.

For the other two terms of \eqref{eq:direct}, Lemma~\ref{lem:direct} gives $|F(1)|\le s^{m/2}$ while Lemma~\ref{lem:products} and the inequality $|\mu|\le\sqrt{s}$ from \eqref{eq:CS} give
\[
 \frac mn|\mu|\Pi|p(0)|\le s^{\frac{m+1}{2}}.
\]
Both are at most $(3/4)^{m/2}\le e^{-m/8}<e^{-4}<1/16$. Combining these estimates with \eqref{eq:case1early} and \eqref{eq:case1tail}, we find that the right-hand side of \eqref{eq:direct} is less than $1/4+3/16=7/16$, a contradiction.

\subsection*{Case 2: \texorpdfstring{$\bm{a\le1/4}$}{a ≤ ¼}}

By \eqref{eq:productnecessary}, $s\ge(1+a^2)^{-1}\ge16/17$. Thus $\sqrt{1-s}\le1/\sqrt{17}<1/4$, $h>3/4$, and inequality \ref{lem:deficit6} in Lemma~\ref{lem:deficit} gives $|\mu|\le5a/4$. Moreover, Lemma~\ref{lem:beta} gives $\beta(1)\le1-9a^2/16$. Using $\Pi\le1$ and $|p(0)|\le\sqrt{e}\,a<(5/3)a$ from Lemma~\ref{lem:products}, the middle term of \eqref{eq:direct} is at most
\begin{equation}\label{eq:case2main}
 |\mu||p(0)|\le\frac{25}{12}a^2\le\frac{25}{192}.
\end{equation}

For $0\le t\le1/2$, the definition of $\beta$ and Lemma~\ref{lem:deficit} give $\beta(t)\le1-3a^2t/2+a^2t^2\le1-a^2t \leq e^{-a^2t}$. The first part of the direct remainder in \eqref{eq:direct} is consequently at most
\begin{equation}\label{eq:case2early}
 c_mma^2\int_0^\infty te^{-(m-1)a^2t/2}\dd < \frac{20}{3}\frac{m^2}{(m-1)^2}\frac1M < \frac{27}{4M} < \frac{27}{56}.
\end{equation}
Here we used $s<1$, $c_m<5/3$, $m^2/(m-1)^2<101/100$, $(20/3)(101/100)<27/4$, and \eqref{eq:M14}.

At both $t=1/2$ and $t=1$ we have $\beta(t)\le1-a^2/2$. Hence, the convexity of $\beta$, as in the proof of Lemma~\ref{lem:beta}, extends this bound to $[1/2,1]$. Since $\int_{1/2}^1t\dd=3/8$, its contribution is at most
\begin{equation}\label{eq:case2tail}
 \frac58M e^{-(m-1)M/(4m)} \le \frac58M e^{-1023M/4094} < \frac{35}{4}e^{-7161/2047} < \frac{35}{128}.
\end{equation}
Indeed $(m-1)/(4m)\ge2046/(4\cdot2047)$; the function $M e^{-1023M/4094}$ decreases for $M>4094/1023$, so \eqref{eq:M14} allows evaluation at $M=14$. The last comparison follows entirely in rational arithmetic from
\[
 e^{7161/2047}>\sum_{j=0}^{7}\frac{(7161/2047)^j}{j!}>32.
\]
Finally, \eqref{eq:Fend} and the endpoint bound $\beta(1) \leq 1-9a^2/16 \leq e^{-9a^2/16}$ give
\begin{equation}\label{eq:case2end}
 |F(1)|\le \beta(1)^{m/2} \leq e^{-9M/32}<e^{-63/16}<\frac1{48},
\end{equation}
where $\sum_{j=0}^{12}(63/16)^j/j!>48$ proves the last step. Adding \eqref{eq:case2main}--\eqref{eq:case2end} gives
\[
 1\le |F(1)|+\frac mn|\mu|\Pi|p(0)|+|R| \leq \frac{25}{192}+\frac{27}{56}+\frac{35}{128}+\frac1{48}
 =\frac{2437}{2688}<1,
\]
contradicting \eqref{eq:direct}.

\subsection*{Case 3: \texorpdfstring{$\bm{s>3/4}$ and $\bm{a>1/4}$}{s > ¾ and a > ¼}}
The remaining range is treated by the centered identity. Here $h>1/2$; the bounds $ah\le|\mu|\le\sqrt{s}<1$ in Lemma~\ref{lem:deficit} and \eqref{eq:CS} imply $a<2$, and inequality \ref{lem:deficit6} in Lemma~\ref{lem:deficit} gives $|\mu|<3a/2$. Lemma~\ref{lem:beta} gives $\beta(1)\le1-a^2/4 \leq e^{-a^2/4}$. Therefore, the identity $\beta(1)=|D(1)|^2+a^2v$ and the covariance bound
\ref{lem:deficit7} in Lemma~\ref{lem:deficit} imply, for $v>0$,
\begin{equation}\label{eq:centerend}
 |D(1)|^n = |D(1)|^2\bigl(\beta(1)-a^2v \bigr)^{\frac{n-2}{2}} \le v\beta(1)^{\frac{n-2}{2}}
 \le v e^{-\frac{(n-2)a^2}{8}}
 \le v e^{-\frac{n-2}{128}}<\frac v{32}.
\end{equation}
For the final comparison, we used $(2048-2)/128>6$ and $e^6>2^6>32$. We first assume that $v>0$ and we estimate $E_c/v$. To this end, we further consider two more subcases.

\smallskip\noindent\emph{Subcase $1/4<a\le1/2$.}
The first bound in \eqref{eq:productnecessary} gives $1-s\le \frac{a^2}{1+a^2}\le a^2$, so $h\ge1-a$. Hence,
\[
 a^2h^3\ge a^2(1-a)^3\ge\frac{27}{1024}.
\]
Indeed, $a^2(1-a)^3$ is concave on $[1/4,1/2]$, so its minimum is at an endpoint, and the smaller endpoint value is $27/1024$. Now, for $t\in[0,1/4]$, $|D(t)| \ge 1-a|\mu|/4 \ge 7/8$. Moreover, $2t\le 1/2 \leq 1-a \leq  h$, so $\beta(t)\le1-2a^2ht+a^2t^2\le1-3a^2ht/2 \leq e^{-3a^2ht/2}$. Use the first term in \eqref{eq:interp1} along with the identity $\int_0^\infty t^2e^{-ct}\dd=2/c^3$ and the bounds $c_m< 5/3$ and $|\mu| \leq a(1+\sqrt{1-s})\leq 3a/2$ to obtain
\begin{align}
 \frac{E_{c,[0,1/4]}}v
 &\le \frac{n|\mu|c_mma^3}{2} \cdot \frac87 \cdot \frac{2}{(3(m-1)a^2h/4)^3} = \frac{512c_m nm|\mu|/a}{189(m-1)^3\,a^2h^3} < \frac{300}{m} < \frac14.\label{eq:case3early}
\end{align}
Here $E_{c,I}$ denotes the integral defining $E_c$ restricted to $I$. For completeness, we note that the function $(m+1)m^2/(m-1)^3$ decreases for $m>1$, since its logarithmic derivative is $-(4m+2)/(m(m-1)(m+1))<0$. Its value at $2047$ is less than $101/100$. Thus the numerical coefficient is bounded by
\[
 \frac{512}{189}\frac53\frac32\frac{101}{100}\frac{1024}{27}<300,
 \qquad \frac{300}{2047}<\frac14.
\]

To control the remaining interval, observe that the same bounds give
\begin{align*}
 1-\beta(1/4)&\ge a^2(h/2-1/16)
     \ge a^2(7/16-a/2)\ge5/256,\\
 1-\beta(1)&\ge a^2h^2
     \ge a^2(1-a)^2\ge9/256.
\end{align*}
Convexity of $\beta$ therefore gives $\beta(t) \le 1-5/256$ on $[1/4,1]$. Moreover, throughout this interval we have $|D(t)|\ge1-a|\mu|\ge1/2$. Since $c_m/2<1$, the first term in \eqref{eq:interp1} yields
\begin{equation}\label{eq:case3tail}
 \frac{E_{c,[1/4,1]}}v \le \frac3{16}n^2e^{-5(n-2)/512}<\frac1{32}.
\end{equation}
Indeed the prefactor is at most $2nm a^3|\mu|\le3nm a^4\le3n^2/16$, and $\int_{1/4}^1t^2\dd\le1$. The function $n^2e^{-5(n-2)/512}$ decreases for $n>1024/5$. At $n=2048$, the exponent exceeds $19$, and
\[
 e^{5(2046)/512}>(5/2)^{19}>6\cdot2048^2,
\]
using $e>1+1+1/2=5/2$. This proves the stated comparison.

\smallskip\noindent\emph{Subcase $a>1/2$.}
We still have $a<2$, $|\mu|<1$, and $h>1/2$. Therefore, on $[0,1/4]$, we have $|D(t)|\ge1-a|\mu|/4>1/2$ and $\beta(t)\le1-3a^2ht/2 \leq e^{-3a^2ht/2}$. Moreover, we also have $a^2h^3> 1/2^{2+3}= 1/32$. The calculation in \eqref{eq:case3early} now becomes in this case
\begin{equation}\label{eq:case4early}
 \frac{E_{c,[0,1/4]}}v
 \le\frac{128c_m}{27}\frac{nm}{(m-1)^3}
          \frac{|\mu|/a}{a^2h^3}
 <\frac{400}{m}<\frac14.
\end{equation}

At $t=1$, Lemma~\ref{lem:beta} and $h>1/2$ gives $\beta(1)\le1-a^2h^2\le1-a^2/4$, while at $t=1/4$ the same lemma together with $s\le2h$ gives $\beta(1/4)\le1-a^2h/2+a^2s/16\le1-3a^2h/8\le1-3a^2/16$. Hence $\beta(t)\le1-3a^2/16$ at $t=1/4$ and at $t=1$. Convexity therefore gives $\beta\le1-3/64$ on $[1/4,1]$. Here, a lower bound for $|D|$ is unnecessary as we will use the second estimate in \eqref{eq:interp1}. Its constant satisfies $c_{m,2}<e<3$ by the same exponential estimate used in Case~1. Since $a<2$ and $|\mu|<1$, we thus find
\begin{equation}\label{eq:case4tail}
 \frac{E_{c,[1/4,1]}}v
 \le12n^3e^{-3(n-3)/128}<\frac1{32}.
\end{equation}
Indeed, the prefactor is bounded by $na|\mu|\,3m(m-1)a^2/2<12n^3$, and the integral of $t^2$ is at most one. The function $n^3e^{-3(n-3)/128}$ decreases for $n>128$. At $n=2048$, $3(2045)/128>47$ and $e^{3(2045)/128} > 2^{47}>384\cdot2048^3$, as required.

Combining \eqref{eq:centerend} with either \eqref{eq:case3early}--\eqref{eq:case3tail} or \eqref{eq:case4early}--\eqref{eq:case4tail}, depending on the cases, we find that $|D(1)|^n+E_c<5v/16$. Lemma~\ref{lem:products} bounds the other term of the centered inequality by
\[
 |\mu|\Pi|p(0)|\le|\mu|=\sqrt{s-v}
 \le\sqrt{1-v}\le1-v/2.
\]
Hence, putting everything together, Equation~\eqref{eq:centered} would now imply $1<1-3v/16$, a contradiction.

Finally, if $v=0$, all $q_j$ equal $\mu$. Identity~\ref{lem:deficit1} in Lemma~\ref{lem:deficit} then gives $a\mu=1$, hence $D(1)=0$ and $F=D^m$. Lemma~\ref{lem:centered} and the inequality $|p(0)|\le1$ instead give $1\le|\mu|\Pi|p(0)|\le|\mu|^{m+1}<1$, because $|\mu|<1$ by \eqref{eq:counter}. Therefore, this rules out the last possibility and proves the theorem.
\end{proof}

\section{Explicit exclusions in finite degree}\label{sec:finitebounds}

The estimates developed in Section \ref{sec:large} were intentionally crude, since the only requirement there was that they decay sufficiently rapidly with respect to the degree to produce a contradiction when $n$ is large. However, for $n \leq 2047$, a more refined analysis is required. The purpose of the two remaining sections is to retain and exploit additional information. Although the overall approach is elementary, the individual formulas are difficult to motivate when viewed in isolation. We therefore begin by outlining the full strategy. Throughout Sections~\ref{sec:finitebounds}--\ref{sec:computer}, the normalized counterexample assumptions \eqref{eq:normal} and \eqref{eq:counter} remain in force, and we always assume that $n\geq4$.

\subsection{The principle}\label{subsec:principle}

\subsubsection*{Shortcut variables and a three-dimensional search space}

Introduce the shortcut variables
\begin{equation}\label{eq:coords}
 A:=a^2,\qquad r:=|\mu|^2,\qquad v=\E\bigl(|q-\mu|^2\bigr)=s-r,
 \qquad b:=2a\re(\mu),
\end{equation}
the square of the prescribed zero, the squared modulus of the mean and the variance of the reciprocal points $q_j$ from \eqref{eq:moments}, and twice the real part of $a\mu$. In particular, we have $\sqrt{A}=a$, $\sqrt{r}=|\mu|$, $\re (\mu)=b/(2a)$, and $(2\log n-1)/n<A\le m$. We also write $k:=(m-1)/m$, so that $\kappa\le k$ by Lemma~\ref{lem:critical}.

Every quantity produced by the estimates of Sections~\ref{sec:normalization} and~\ref{sec:integrals} is then either an exact function of $A,r,v,b$ and $n$, or is bounded above by one. The first alternative covers the two functions
\begin{equation}\label{eq:betaD}
 \beta(t)=1-bt+Ast^2,\qquad
 D(t)=1-ta\mu,\qquad |D(t)|^2=1-bt+Art^2=\beta(t)-Avt^2
\end{equation}
in Section~\ref{sec:integrals}, and also $\mu$ itself, since $|\mu|^2=r$ and $\re(\mu)=b/(2a)$ determine $\mu$ once the whole root configuration has been conjugated, if necessary, so that $\im(\mu)\ge0$. The second alternative covers the remaining ingredients: the critical-point moment $\kappa$ through $\kappa\le k$ (Lemma~\ref{lem:critical}), the products $\Pi\,|p(0)|$ through Lemma~\ref{lem:products}, the endpoint value $|F(1)|$ through \eqref{eq:Fend}, and the centered coefficients of Subsection~\ref{subsec:centered} through $|e_j(\omega)|\le\alpha_j(Av)^{j/2}$.

Nonetheless, in this section, $b$ will \emph{not} be considered a fourth coordinate. Indeed, Proposition~\ref{prop:restrictions} confines it to an interval determined by $(A,r,v)$ and $n$, and Remark~\ref{rem:boxeval} eliminates it, monotonically for $\beta$ and $|D|$ and by an explicit enclosure for $\mu$. Every bound we write below is therefore, in the end, an explicit function of $(A,r,v)$ and $n$, and a counterexample in degree $n$ thus produces a point of the compact search space
\begin{equation}\label{eq:searchspace}
 \mathcal D_n:=[0,m]\times[0,1]\times[0,1]
 \quad\text{in the coordinates }(A,r,v).
\end{equation}
The whole finite part of the proof consists in showing that no point of $\mathcal D_n$ can come from a counterexample, for each $n$ between $4$ and $2047$.

\subsubsection*{Step 1: restrictions among the parameters}
Lemmas~\ref{lem:critical}--\ref{lem:products} and~\ref{lem:beta} force the coordinates of a counterexample to satisfy a list of explicit inequalities among themselves. These are collected in Proposition~\ref{prop:restrictions}, and they already remove a large portion of $\mathcal D_n$.

\subsubsection*{Step 2: tests}
The two identities of Section~\ref{sec:integrals} obtained in Lemmas~\ref{lem:direct} and \ref{lem:centered}, together with a third one obtained below in Subsection~\ref{subsec:endpoint}, each state that the number $1$ is at most a sum of terms built from the counterexample's configuration. Bounding every such term by an explicit elementary function of $(A,r,v,n)$ turns each identity into an inequality of the shape
\begin{equation*}\label{eq:testshape}
 1\le\Phi(A,r,v,n),
\end{equation*}
where $\Phi$ is a completely explicit expression: finitely many arithmetic operations, square roots, powers, and one integral of an elementary function over $[0,1]$. Three such functions $\Phi$ are produced in Subsections~\ref{subsec:direct}--\ref{subsec:endpoint}. In the following, an inequality of this form is called a \emph{test}. A counterexample must \emph{satisfy} the test, so showing that $\Phi<1$ somewhere proves that no counterexample lives there.

\subsubsection*{Step 3: boxes}
Both the restrictions of Step~1 and the tests of Step~2 involve only finitely many elementary operations. Hence, if $\mathcal B\subset\mathcal D_n$ is a closed box $I_A\times I_r\times I_v$ with rational endpoints, a single finite computation with intervals produces a number $\Phi_{\mathcal B}$ with $\Phi(A,r,v,n)\le\Phi_{\mathcal B}$ for \emph{every} point of $\mathcal B$. The box is then discarded if
\begin{itemize}
\item some restriction of \emph{Step~1} fails at every point of $\mathcal B$, or
\item $\Phi_{\mathcal B}<1$ for one of the three tests.
\end{itemize}
If a finite family of boxes covers $\mathcal D_n$ and every member is discarded, then no counterexample can exist in degree $n$.

\subsubsection*{Step 4: carrying out the computation}
The subdivision and the evaluation of the bounds are done by computer, in Section~\ref{sec:computer}. All arithmetic is performed on intervals with exact dyadic rational endpoints, rounded outwards, so that each computed $\Phi_{\mathcal B}$ is a rigorous upper bound; and the subdivision is recorded as a binary tree, so that the covering of $\mathcal D_n$ is verified by induction rather than assumed. Since the bounds are moreover monotone in $n$, one box can serve a whole block of consecutive degrees (see Subsection~\ref{subsec:blocks}), which is what keeps the range up to $2047$ tractable. One situation escapes this scheme and is treated analytically in Subsection~\ref{sec:boundary}, namely one neighborhood in degree four where all three tests lose their margin at once.

\subsection{Restrictions}\label{subsec:restrictions}

We first proceed with Step~1 by proving the following result. To this end, set
\begin{equation*}
K:=\sqrt{A}\left(\frac{s(n-A)}{m}\right)^{\!m/2}\!\!,\qquad L:=\min\Bigl\{(1-s)\sqrt{mA},\, \sqrt{\smash{(1-s)(A-1/m)}\vphantom{A}}\Bigr\}.
\end{equation*}

\begin{proposition}
\label{prop:restrictions}
Let $n\ge4$ and assume \eqref{eq:normal} and \eqref{eq:counter}. Then the parameters \eqref{eq:coords} satisfy
\vspace{-5pt}
\begin{multicols}{3}
\begin{enumerate}
    \itemsep=3.5pt

    \itemtag{(1)} $0\le v\le s<1$;\label{prop:restrictions1}

    \itemtag{(4)} $|\hspace{-1pt}\im(\mu)|\le L$;\label{prop:restrictions4}
    
    \itemtag{(2)} $ns(1+A)\ge n^{2/m}m$;\label{prop:restrictions2}
    
    \itemtag{(5)} \makebox*{Box}{$\max\bigl\{A+r-L^2,\, 1+Ar -kv\bigr\}
   \le b \le 2\sqrt{Ar}$.\!\!\!\!\!\!\!\!\!\!\!\!\!\!\!\!\!\!\!\!\!\!\!\!\!\!\!\!\!\!\!\!\!\!\!\!\!\!\!\!\!\!\!\!\!\!\!\!\!\!\!\!\!\!\!\!\!\!\!\!\!\!\!\!\!\!\!\!\!\!\!\!\!\!\!\!\!\!\!\!\!\!\!\!\!\!\!\!\!\!\!\!\!\!\!\!\!\!\!\!\!\!\!\!\!}\label{prop:restrictions5}
   
    \itemtag{(3)} $\bigl|\sqrt{A}-\sqrt{r}\bigr|\le L$;\label{prop:restrictions3}

    \item[\,] 
\end{enumerate}
\end{multicols}
\vspace{-10pt}
\end{proposition}

\begin{proof}
Both \ref{prop:restrictions1} and \ref{prop:restrictions2} have already been proved: they follow directly from Lemmas~\ref{lem:deficit} and~\ref{lem:products}.

For \ref{prop:restrictions3} and \ref{prop:restrictions4}, recall from Lemmas~\ref{lem:deficit} and~\ref{lem:critical} that $\sum_j|q_j|^{-2}=m(A+\kappa)\le mA+m-1$. Since $|q_j|<1$ by \eqref{eq:counter}, we get $|q_j|^{-2}+m-1<\sum_{j}|q_j|^{-2}\le mA+m-1$, so $|q_j|^{-2}\le mA$ for every $j$. The numbers $\delta_j=1-|q_j|^2$ have mean $1-s$, so the identity $a-\overline\mu=\E(\delta/q)$ from the proof of Lemma~\ref{lem:deficit} gives $|a-\overline\mu|\le\E(\delta/|q|)\le(1-s)\sqrt{mA}$. The second member of $L$ is \ref{lem:deficit8} in Lemma~\ref{lem:deficit}. Now $|a-\overline\mu|\le L$ yields \ref{prop:restrictions3} and \ref{prop:restrictions4} by the reverse triangle inequality and by taking imaginary parts.

Lastly, for \ref{prop:restrictions5}, take the squared moduli in $|a-\overline\mu|\le L$ to obtain $A+r-b\le L^2$, i.e., the first lower bound. The second is obtained by expanding \ref{lem:deficit7} in Lemma~\ref{lem:deficit}, which gives $1-b+Ar\le k(s-r) = kv$. The upper bound is simply $\re(\mu)\le|\mu|$.
\end{proof}

\begin{remark}\label{rem:boxeval}
Proposition~\ref{prop:restrictions} allows the tests below, stated in terms of $\beta$ and $D$, to be evaluated on an $(A,r,v)$-box. Indeed, \ref{prop:restrictions5} confines $b$ to $[b_0,2\sqrt{Ar}]$ with $b_0:=\max\{A+r-L^2,\,1+Ar-kv\}$, and $\beta(t)$ and $|D(t)|^2$ decrease in $b$ for $t>0$ by \eqref{eq:betaD}. Together with \eqref{eq:betachord}, this gives, for $0\le t\le1$,
\begin{equation}\label{eq:explicitbounds}
 0\leq \beta(t) \le \min\{1, 1-b_0t+Ast^2\},\qquad
 |D(t)|^2 \le 1-b_0t+Art^2,\qquad
 |D(t)|\ge1-t\sqrt{Ar},
\end{equation}
and $\gamma(t)\le s\bigl(A(1-t)^2+k\bigr)$ by Lemma~\ref{lem:direct}, while $\mu$ is enclosed by $\re(\mu)\in[b_0/(2a),\sqrt r]$ and $|\im(\mu)|\le\min\bigl\{\sqrt{r-\re(\mu)^2},\,L\bigr\}$. On a box, \ref{prop:restrictions1}, \ref{prop:restrictions3} and \ref{prop:restrictions5} also confine $r$: with $v=s-r$, the second lower bound in \ref{prop:restrictions5} against its upper bound reads $(A+k)r-2a\sqrt r+1-ks\le0$, so
\[
\max\{a-L,\,x_-\}\le\sqrt r\le\min\{\sqrt s,\,a+L,\,x_+\},\qquad x_\pm:=\frac{a\pm\sqrt{k((A+k)s-1)}}{A+k},
\]
a range that is empty if $(A+k)s<1$. Both entries of $b_0$ increase with $r$, so $b_0$ may be evaluated at the smallest such $r$.
\end{remark}

\subsection{The direct test}\label{subsec:direct}

Our first test is the direct identity of Lemma~\ref{lem:direct}, with every quantity replaced by its corresponding expression in the coordinates \eqref{eq:coords}.

\begin{proposition}
\label{prop:direct}
Every hypothetical counterexample satisfies
\begin{equation}\label{eq:directtest}
 1 \le \min\bigl\{ks,\, \beta(1)\bigr\}^{\frac{m}{2}}
   +\frac mn K\sqrt{r}
   +c_mmAs\int_0^1 t
     \min\bigl\{\beta(t),\, s\bigl(A(1-t)^2+k\bigr)\bigr\}^{\!\frac{m-1}2}\dd.
\end{equation}
\end{proposition}

\begin{proof}
Insert $|\mu|=\sqrt{r}$, $\Pi|p(0)|\le K$ and $\kappa\le k$ into \eqref{eq:direct}, \eqref{eq:remainder} and \eqref{eq:Fend}.
\end{proof}

\subsection{The centered test}\label{subsec:centered}

The centered identity of Lemma~\ref{lem:centered} is more delicate, but also much stronger, because its entire error is proportional to the variance $v$. Both of its forms below have the same shape, and we begin by isolating that shape in the following proposition. We also recall that $E_c$ is the error term of the centered identity, defined in \eqref{eq:defEC}.

\begin{proposition}
\label{prop:centered}
Let $\mathcal E\ge0$ be such that $|D(1)|^n+E_c\le v\mathcal E$. Then every counterexample satisfies
\begin{equation}\label{eq:centeredtest}
 1\le K\sqrt{r}+v\mathcal E .
\end{equation}
Hence, a region is excluded if $K\sqrt{r}+v\mathcal E<1$, and in particular if $2\mathcal E\le K$ holds throughout.
\end{proposition}

\begin{proof}
Inequality \eqref{eq:centeredtest} is \eqref{eq:centered} together with $|\mu|=\sqrt{r}$ and $\Pi|p(0)|\le K$. For the second assertion, note that $\sqrt{s}+\sqrt{r} \le 1+1 = 2$, so $\sqrt{s}-\sqrt{r}=v/(\sqrt{s}+\sqrt{r})\ge v/2$. Hence, $2\mathcal E\le K$ and Lemma~\ref{lem:products} give $K\sqrt{r}+v\mathcal E\le K\sqrt{r}+Kv/2\le K\sqrt{s}\le s^{n/2}<1$, contradicting \eqref{eq:centeredtest}.
\end{proof}

It remains to produce admissible $\mathcal E$. The first one is obtained by integrating the two interpolation bounds of Lemma~\ref{lem:centered} directly.
 
\begin{corollary}
\label{cor:Tint}
Let $d\ge|D(1)|$. Then we have $|D(1)|^n+E_c \leq v\mathcal{E}_{\mathrm{int}}$, where
\begin{equation}\label{eq:Tint}
 \mathcal E_{\mathrm{int}}:=k\,d^{\,n-2}
 +\frac{nmA^{3/2}\sqrt{r}}{2}\int_0^1 t^2\beta(t)^{\frac{m-2}2}\min\!\bigg\{
 \frac{c_m\sqrt{\beta(t)}}{|D(t)|},\,
 c_{m,2}(m-1)\bigg\}\dd,
\end{equation}
the first entry being omitted if $D(t)=0$. In particular, we may use $d=\sqrt{1-b+Ar}$.
\end{corollary}
 
\begin{proof}
By \ref{lem:deficit7} in Lemma~\ref{lem:deficit}, $|D(1)|^2=|1-a\mu|^2\le kv$, so $|D(1)|^n \le kvd^{\,n-2}$. Since $na|\mu|=n\sqrt{Ar}$, \eqref{eq:interp1} bounds $E_c$ by $v$ times the integral in \eqref{eq:Tint}. The choice of $d$ is simply \eqref{eq:betaD}.
\end{proof}
 
In low degrees, bounding the whole interpolation error $|D(1)|^n+E_c$ at once loses too much information. We therefore use an expansion of $F$ in terms of the powers of $t$ and bound each coefficient separately to produce a second possible choice for $\mathcal{E}$. For the following, we define
\[
e_j(\omega) := e_j(\omega_1,\ldots,\omega_m),\qquad J_j(d):=\frac1{n\binom mj}\sum_{\ell=0}^{m-j}\bigg( \genfrac{}{}{0pt}{}{j+\ell}{j} \bigg)d^{\,\ell},
\]
where $e_j(\omega_1,\ldots,\omega_m)$ is the $j$th elementary symmetric polynomial of the $\omega_j$ defined in \eqref{eq:defw}. Then, by the properties $\sum_j\omega_j=0$ and \eqref{eq:moments}, we have $e_0(\omega)=1$, $e_1(\omega)=0$, $\E\bigl(|\omega|^2\bigr)=Av$. Moreover, define
\begin{equation}\label{eq:alpha}
 \alpha_2:=\frac m2,\qquad \alpha_m:=1,\qquad
 \alpha_j:=\left(\frac mj\right)^{\!j/2}\!
          \left(\frac m{m-j}\right)^{\!(m-j)/2}\quad(2<j<m).
\end{equation}
 
\begin{lemma}
\label{lem:coefficients}
Using the above definitions, we have
\begin{equation}\label{eq:Fcoeff}
     F(t)=\sum_{j=0}^me_j(\omega)\,t^jD(t)^{m-j}.
\end{equation}
Moreover, $|e_j(\omega)|\le\alpha_j(Av)^{\smash{j/2}}$ for $2\le j\le m$, and, for $0\le j\le m$ and any $d\ge0$ with $|D(1)|\le d$,
\begin{equation*}\label{eq:Jcoeff}
 \int_0^1t^j|D(t)|^{m-j}\dd\le J_j(d).
\end{equation*}
\end{lemma}
 
\begin{proof}
Writing $1-atq_j=D(t)+t\omega_j$ and expanding the product gives \eqref{eq:Fcoeff}. Moreover, since $\sum_j\omega_j=0$, we have $e_2(\omega)=-\frac12\sum_j\omega_j^2$, whence $|e_2(\omega)|\le\frac12\sum_j|\omega_j|^2=\frac m2Av$ and $\alpha_2=m/2$. For the remaining $j$, AM--GM and $\sum_j\omega_j=0$ give, for every $z\in\C$,
\[
 \Bigl|\prod_j(1+z\omega_j)\Bigr|
 \le \Bigl(\E\bigl|1+z\omega\bigr|^2\Bigr)^{\smash{m/2}}
 =\bigl(1+|z|^2Av\bigr)^{m/2},\vspace{-3pt}
\]
so Cauchy's estimate on $|z|=R$ for the coefficient of $z^j$ yields $|e_j(\omega)|\le R^{-j}(1+R^2Av)^{m/2}$. For $j<m$ and $Av>0$, this is minimized at $R^2=j/((m-j)Av)$, which gives \eqref{eq:alpha}. The case $j=m$
is then simply AM--GM applied to $\prod_j|\omega_j|$.
 
For the integral estimate, write $D(t)=(1-t)+tD(1)$, so that $|D(t)|\le1-t+td$. Expanding $t^j(1-t+td)^{m-j}$ and using the beta integral $\binom{m-j}{\ell}\int_0^1 t^{j+\ell}(1-t)^{m-j-\ell}\dd = n^{-1}\binom{m}{j}^{-1}\binom{j+\ell}{j}$ gives the stated value.
\end{proof}
 
\begin{corollary}
\label{cor:Tcoeff}
Let $d\ge|D(1)|$. Then we have $|D(1)|^n+E_c \leq v\mathcal{E}_{\mathrm{coef}}$, where
\begin{equation}\label{eq:Tcoeff}
 \mathcal E_{\mathrm{coef}}:=kd^{\,n-2}
 +nA\sqrt{Ar}\sum_{j=2}^m\alpha_j(Av)^{\frac{j-2}2}J_j(d).
\end{equation}
\end{corollary}
 
\begin{proof}
By \eqref{eq:Fcoeff} and $e_1(\omega)=0$, we have $F(t)-D(t)^m = \sum_{j\ge2}e_j(\omega)t^jD(t)^{m-j}$, so Lemma~\ref{lem:coefficients} gives $\int_0^1|F(t)-D(t)^m|\dd\le\sum_{j\ge2}\alpha_j(Av)^{j/2}J_j(d)$. Multiplying by $na|\mu|=n\sqrt{Ar}$ and factoring out $Av$ bounds $E_c$ by $v$ times the sum in \eqref{eq:Tcoeff}; the term $kd^{\,n-2}$ handles $|D(1)|^n$ as in Corollary~\ref{cor:Tint}.
\end{proof}
 
\subsection{The endpoint test}\label{subsec:endpoint}
 
The two previous tests obtained estimations by integrating over the entire interval from the prescribed zero $a$ to the centroid $0$. In low degrees, a better bound can be obtained by stopping earlier, namely at
$a(1-T)$ for some $0<T\le1$, and to combine the resulting integral with a linear relation among the centered coefficients $e_j(\omega)$.

In the following, we define, for $0<T\leq 1$ and $0\leq j \leq m$, the quantities
\begin{equation*}\label{eq:CT}
 C_T:=\frac nm+A\!\left((1-T)^2+\frac{1-2T}m\right),\qquad B_j:=(-1)^j (a\mu)^{m-1-j}\bigl(m-j-ma\mu\bigr),
\end{equation*}
where $C_T$ is the mean squared distance from $a(1-T)$ to the remaining roots $z_2,z_3,\dots,z_n$. 
 
\begin{proposition}
\label{prop:dual}
For any $0<T\le1$, we have
\begin{equation*}\label{eq:masteridentity}
    \sum_{j=0}^mB_je_j(\omega)=0,\qquad
    \int_0^TF(t)\dd=\sum_{j=0}^me_j(\omega)\Bigl(T^{j+1}J_j(D(T))+\lambda B_j\Bigr) \quad(\lambda\in\C),
\end{equation*}
and every counterexample satisfies, for every $\lambda\in\C$,
\begin{equation}\label{eq:dual}
 \bigl|TJ_0\bigl(D(T)\bigr)+\lambda B_0\bigr|
 \le\frac{T(sC_T)^{m/2}}n
 +\sum_{j=2}^m\alpha_j(Av)^{j/2}
   \bigl|T^{j+1}J_j\bigl(D(T)\bigr)+\lambda B_j\bigr| .
\end{equation}
\end{proposition}
 
\begin{proof}
Write $z_T:=a(1-T)$. By \eqref{eq:normal}, the roots other than $a$ have sum $-a$ and squared-modulus sum $n-A$, so
\[
 \frac1m\sum_{i=2}^n|z_i-z_T|^2=\frac{n-A+2az_T+mz_T^2}m=C_T,
\]
and AM--GM gives $|p(z_T)|\le aTC_T^{m/2}$. Since $F(t)=p'(a(1-t))/p'(a)$ and $p(a)=0$, integration gives $p(z_T)=-ap'(a)\int_0^TF\dd$. Using $|p'(a)|=n/\Pi$ and $\Pi\le s^{m/2}$ from Lemma~\ref{lem:products} then yields
\begin{equation}\label{eq:endpointintegral}
 \bigg|\int_0^TF(t)\dd\bigg| \le \frac{T(sC_T)^{m/2}}n .
\end{equation}
Define $H(z) := \prod_j(z-\omega_j)$. Since $\omega_j=a(\mu-q_j)$, we have $H(a\mu)=\prod_j(aq_j)\ne0$. Also, by $\E(q^{-1})=a$ from \ref{lem:deficit1} in Lemma~\ref{lem:deficit}, $H'(a\mu)/H(a\mu)=\sum_j(aq_j)^{-1}=m$. Expanding this function gives $H(z)=\sum_{j=0}^m(-1)^je_j(\omega)z^{m-j}$, which implies that
\[
H'(z)-mH(z)=\sum_{j=0}^m(-1)^je_j(\omega)\, z^{m-1-j}(m-j-mz).
\]
Hence, $0= H'(a\mu)-mH(a\mu)=\sum_{j=0}^mB_je_j(\omega)$, which is the first identity.

For the second, substitute $t=T\tau$ and note that $D(T\tau)=(1-\tau)+\tau D(T)$. Hence, the binomial identity and the beta-integral computation of Lemma~\ref{lem:coefficients} yield
\begin{align}\label{eq:intJ}
    \int_0^T\! t^jD(t)^{m-j}\dd = T^{j+1}\sum_{\ell=0}^{m-j} D(T)^\ell \bigg( \genfrac{}{}{0pt}{}{m-j}{\ell} \bigg) \int_0^1\! \tau^{j+\ell}(1-\tau)^{m-j-\ell}\,\mathrm{d}\tau =T^{j+1}J_j\bigl(D(T)\bigr).
\end{align}
Therefore, integrating \eqref{eq:Fcoeff} over $[0,T]$ and adding $\lambda\sum_jB_je_j(\omega)=0$ proves the identity. Finally, to obtain the last inequality, isolate the term $j=0$, use $e_0(\omega)=1$ and $e_1(\omega)=0$, and apply \eqref{eq:endpointintegral}.
\end{proof}
 
The previous proposition thus leads to the following test.
 
\begin{corollary}
\label{cor:shifted}
Let $0<T\le1$ and let $d_T\ge|D(T)|$. Every counterexample satisfies
\begin{equation}\label{eq:shifted}
 1\le\sqrt{Ar}\,T(sC_T)^{m/2}+d_T^{\,n} + n\sqrt{Ar}\sum_{j=2}^m\alpha_j(Av)^{j/2}T^{j+1}J_j(d_T).\vspace{-2pt}
\end{equation}
In particular, one may take $d_T^2=1-bT+ArT^2$ or $d_T=1-T+T\sqrt{kv}$.
\end{corollary}
 
\begin{proof}
Integrating $D^{m}$ and using \eqref{eq:intJ} yields $na\mu\,TJ_0(D(T))=na\mu\int_0^TD^m\dd=1-D(T)^n$. Taking $\lambda=0$ in \eqref{eq:dual}, multiplying by $na|\mu|=n\sqrt{Ar}$ and using the triangle inequality gives \eqref{eq:shifted}. The first value for $d_T$ is \eqref{eq:betaD}, and the second follows from $D(T)=1-T+TD(1)$ with $|D(1)|^2\le kv$, which is inequality \ref{lem:deficit7} in Lemma~\ref{lem:deficit}.
\end{proof}

\subsection{The quartic test}
\label{sec:boundary}
 
While most degrees are correctly treated by the tests obtained in this section, degree 4 has one neighborhood where the three critical points are very close to the centroid of the roots relative to the prescribed zero (i.e., $v\ll r$), in which all three tests lose their margin simultaneously. However, a second-order expansion provides one last test, precise in this regime, which allows us to complete the proof. 

\begin{proposition}\label{prop:quarticlocal}
Let $p$ be a quartic with root centroid $b_p$ and let $a\ne b_p$ be a zero of $p$. Then
\begin{equation}\label{eq:localhyp}
 \max_{1\le j\le3}\frac{|\zeta_j-b_p|}{|a-b_p|}\le\frac1{20} \qquad\implies\qquad \frac{\sigma_2^2(p)}3\sum_{j=1}^3|a-\zeta_j|^{-2}\ge1 .
\end{equation}
\end{proposition}

\begin{proof}
Both of the inequalities in \eqref{eq:localhyp} are invariant under affine changes of variable, so we may assume $b_p=0$ and $a=1$. Then $\sum_j\zeta_j=0$ and $|\zeta_j|\le\frac1{20}$. Put
\[
 V:=\sigma_2^2(p)=\frac14\sum_i|z_i|^2,\quad
 S:=\frac13\sum_j|1-\zeta_j|^{-2},\quad
 \tau:=\sum_j|\zeta_j|^2\le\frac{3}{20^2},\quad
 R:=\sum_j\re(\zeta_j^2),
\]
so that the second inequality in \eqref{eq:localhyp} becomes $VS\ge1$ and $R\in[-\tau,\tau]$.

Since $\sum_j\zeta_j=0$, we have $p'(t)=4t^3-2t\sum_j\zeta_j^2-4\zeta_1\zeta_2\zeta_3$. Integrating over $[0,1]$ and using $p(1)=0$ gives $-p(0)=1-\sum_j\zeta_j^2-4\zeta_1\zeta_2\zeta_3$. By AM--GM, and since $|\zeta_1\zeta_2\zeta_3| \le \frac1{20}|\zeta_1\zeta_2|\le\frac\tau{40}$,
\begin{equation}\label{eq:V}
     V^2\ge|z_1z_2z_3z_4|=|p(0)|\ge1-R-\frac{\tau}{10}.
\end{equation}
Define $w:=z^2/(1-z)$, so that $(1-z)^{-1}=1+z+w$, $w-z^2=zw$ and $|w|\le |z|^2/(1-|z|)$. Hence,
\[
 |1-z|^{-2}\ge|1+z|^2+2\re(w)+2\re(\bar zw)
 \ge1+2\re(z)+|z|^2+2\re(z^2)-\frac{4|z|^3}{1-|z|},
\]
where the last term is at most $\frac{4}{19}|z|^2$. Averaging over $z=\zeta_j$ and using $\sum_j\zeta_j=0$ thus yields
\begin{equation}\label{eq:S}
    S\ge1+\frac{\tau+2R}3-\frac{4\tau}{57} \geq 1+\frac{2R}{3}+\frac{\tau}{4}.
\end{equation}
Both of the lower bounds \eqref{eq:V} and \eqref{eq:S} are positive, and $(1+y)^2\ge1+2y$, so
\begin{align*}
    V^2S^2 &\ge \Bigl(1-R-\frac{\tau}{10}\Bigr)\Bigl(1+\frac{4R}3+\frac{\tau}{2}\Bigr)  = 1+\frac{R}{3}+\frac{2}{5}\tau-\frac{4}{3}R^2 -\frac{19}{30}R\tau -\frac{1}{20}\tau^2 \\
    &\geq 1-\frac{\tau}{3}+\frac{2}{5}\tau-\frac{4\cdot 3}{3\cdot20^2}\tau -\frac{19\cdot 3}{30\cdot 20^2}\tau -\frac{1\cdot 3}{20\cdot 20^2}\tau = 1+\frac{1237\tau}{3\cdot20^{3}} \geq 1,
\end{align*}
where we used the fact that $|R|\leq \tau \leq 3/20^2$.
\end{proof}
 
It remains to write the hypothesis \eqref{eq:localhyp} in the coordinates \eqref{eq:coords}. Assume \eqref{eq:normal} and \eqref{eq:counter} with $n=4$. The proof of Proposition~\ref{prop:restrictions} gives $|q_j|^{-2}\le3A$ for every $j$, and the mean minimizes the mean square. Hence, since $\E(q^{-1})=a$, we have for $r>0$
\[
 \kappa=\E\bigl|q^{-1}-a\bigr|^2\le\E\bigl|q^{-1}-\mu^{-1}\bigr|^2
 =\E\bigg(\frac{|q-\mu|^2}{|q|^2r}\bigg)\le\frac{3Av}r .
\]
Furthermore $|\zeta_j|^2=\bigl|\sum_{\ell\ne j}\zeta_\ell\bigr|^2
\le2\sum_{\ell\ne j}|\zeta_\ell|^2=2(3\kappa-|\zeta_j|^2)$, whence
$|\zeta_j|^2\le2\kappa$ for every $j$ and
\begin{equation}\label{eq:recognize_local}
 \max_j\frac{|\zeta_j|}a\le\sqrt{\frac{6v}r}.
\end{equation}
Since $r=s-v$, the condition $6v/r\le20^{-2}$ reads $2401v\le s$. Hence, it forces $r>0$ and, by \eqref{eq:recognize_local}, implies \eqref{eq:localhyp}. In the normalization \eqref{eq:normal} one has $\sigma_2(p)=1$ and $\sum_j|a-\zeta_j|^{-2}=3s$, so Proposition~\ref{prop:quarticlocal} would give $s\ge1$, contradicting \eqref{eq:counter}. Hence any region of $\mathcal D_4$ on which $2401v\le s$ holds is automatically excluded. This conclusion will allow us to prove the main theorem for $n=4$, thus completing the proof.

\section{Finite exact verification}\label{sec:computer}
 
Section~\ref{sec:finitebounds} reduces a hypothetical counterexample of degree $4\le n\le2047$ to a point of the compact search space $\mathcal D_n$ in \eqref{eq:searchspace} satisfying the restrictions of Proposition~\ref{prop:restrictions} and all three tests (four in the case $n=4$). We now carry out Steps~3 and~4 of Subsection~\ref{subsec:principle}, namely a finite covering of $\mathcal D_n$ by boxes, on each of which one of those conditions fails.
 
\subsection{Boxes, certificates, and coverage}\label{subsec:boxes}

Call a point of $\mathcal D_n$ \emph{feasible} if it satisfies every conclusion of Proposition~\ref{prop:restrictions}, every test of Subsections~\ref{subsec:direct}--\ref{subsec:endpoint}, and, when $n=4$, the inequality $2401v>s$ from Subsection~\ref{sec:boundary}. A closed region $\mathcal B\subset\mathcal D_n$ with rational endpoints --- in practice a box $I_A\times I_r\times I_v$, or the prism of Subsection~\ref{subsec:blocks} --- is \emph{excluded} when a rigorous computation shows that every point of $\mathcal B$ violates at least one of those conditions. The admissible ways of doing so %, which we shall prove below,
are the following:
\begin{itemize}
\item an upper bound over $\mathcal B$ for one side of a restriction in either Proposition~\ref{prop:restrictions} or \eqref{eq:explicitbounds} falls below a lower bound for the other side; in particular a negative upper bound for the nonnegative quantity $\beta(t)$ is impossible, by \eqref{eq:betachord};

\item an upper bound over $\mathcal B$ for the right-hand side of the direct test \eqref{eq:directtest}, a centered test \eqref{eq:centeredtest}, or a shifted endpoint test \eqref{eq:shifted}, is less than one;
\item the non-strict centered criterion $2\mathcal{E}\leq K$ in Proposition~\ref{prop:centered}, or its variant \eqref{eq:cubicmono} below for a prism, holds throughout $\mathcal B$;
\item a lower bound for the left-hand side of  \eqref{eq:dual} exceeds an upper bound of its right-hand side;
\item $n=4$ and $2401v\le s$ on $\mathcal B$, so that Subsection~\ref{sec:boundary} applies.
\end{itemize}
 
\begin{lemma}
\label{lem:certificate}
Let $4\le n\le2047$ and let $\mathcal B_1,\ldots,\mathcal B_N$ be closed regions with $\mathcal D_n\subset\bigcup_i\mathcal B_i$, each of them excluded. Then Theorem~\ref{thm:main} holds in degree $n$.
\end{lemma}
 
\begin{proof}
A counterexample in degree $n$ satisfies \eqref{eq:normal} and \eqref{eq:counter} after normalization, hence determines a feasible point of $\mathcal D_n$ by Subsection~\ref{subsec:principle} and, when $n=4$, Subsection~\ref{sec:boundary}. That point lies in some $\mathcal B_i$, contradicting the exclusion of $\mathcal B_i$.
\end{proof}

The boxes are recorded as a finite binary tree, so that each box can be reconstructed exactly. The root is $\mathcal D_n$, each internal node splits one coordinate interval at its rational midpoint into two children, and each leaf records an exclusion from the list above or is explicitly marked unresolved. By induction on the tree, the leaves cover $\mathcal D_n$. Hence, if no leaf is unresolved and recomputation confirms that every leaf is excluded, the leaves satisfy Lemma~\ref{lem:certificate}. 
 
Two conventions are worth noting. First, the implementation parametrizes the variables $(r,v)$ by
\[
 \eta:=\sqrt{1-s},\qquad f:=\sqrt{v\smash{/}s\vphantom{A}},\qquad\text{so that}\quad s=1-\eta^2,\quad r=s(1-f^2),\quad v=sf^2,
\]
which maps $[0,1]^2$ onto $\{(r,v):r,v\ge0,\ r+v\le1\}$. It makes the delicate boundary $s\to1$ easy to subdivide, and it turns the quartic rule of Subsection~\ref{sec:boundary} into the linear condition $2401f^2\le1$. The part of $\mathcal D_n$ left outside that image, where $r+v>1$, needs no subdivision since it is excluded by $s<1$. Second, $\mathcal D_n$ is taken as a closed set, so it contains boundary points that no actual configuration realizes. However, this is irrelevant since such points will be excluded like any others.
 
\subsection{Certified arithmetic}\label{subsec:arithmetic}

The exclusions of Subsection~\ref{subsec:boxes} require bounds that hold rigorously at every point of a box. Floating-point evaluation cannot guarantee this, since a rounding error may move a computed value to the wrong side of a comparison. The checker therefore replaces every quantity by a dyadic interval $[\ell/S_0,\,u/S_0]$ with $S_0=2^{96}$ and integers $\ell\le u$ of arbitrary size, and rounds every operation outwards so that the result contains all values the quantity takes on the box.

Concretely, a rational $p/q$ with $q>0$ is enclosed by $\bigl[\lfloor S_0p/q\rfloor,\lceil S_0p/q\rceil\bigr]/S_0$. Addition and negation act exactly on the endpoints. A product is enclosed by the extreme values of the four endpoint products, rounded outwards, and a quotient likewise, provided the denominator interval does not contain zero. The square root of a nonnegative interval is $\bigl[\lfloor\sqrt{\ell S_0}\rfloor,\lceil\sqrt{uS_0}\rceil\bigr]/S_0$, computed with the integer square root. Powers of nonnegative intervals follow by repeated multiplication and a square root if needed. Complex quantities are enclosed by rectangles, that is, pairs of real intervals, and a modulus is enclosed from both sides as the square root of the enclosure of the squared modulus. Therefore, no floating-point number ever enters a comparison that accepts a box.
 
By induction on each arithmetic expression, the computed interval contains the true value at every point of the box. Hence an upper endpoint below one proves the inequality required by a test, and a lower endpoint above an upper endpoint proves two enclosures incompatible.
 
\subsection{Certified integration}\label{subsec:quadrature}
 
The direct and centered tests contain integrals over $[0,1]$. All of them are bounded rigorously by the same positive-weight chord rule, which requires only convexity.
 
\begin{lemma}\label{lem:chord}
Let $g$ be nonnegative and convex on $[\ell,h]$ with $0\le\ell<h$. Then
\begin{equation*}\label{eq:chordweights}
 \int_\ell^ht^jg(t)\dd\le w_\ell\,g(\ell)+w_h\,g(h),\qquad
 w_\ell:=\int_\ell^h t^j\frac{h-t}{h-\ell}\dd,\quad
 w_h:=\int_\ell^h t^j\frac{t-\ell}{h-\ell}\dd,
\end{equation*}
for any integer $j\ge0$. Moreover $\beta^\rho(t)$ and $\gamma^\rho(t)$ are convex on $[0,1]$ for every $\rho\ge1/2$.
\end{lemma}
 
\begin{proof}
Multiply the chord inequality $g(t)\le\frac{h-t}{h-\ell}\,g(\ell)+\frac{t-\ell}{h-\ell}\,g(h)$ by $t^j\ge0$ and integrate. For the convexity, note that
\[
 \sqrt{\smash{\beta(t)}\vphantom{A_2}}=\frac1{\sqrt m}\bigl\|(1-atq_1,\ldots,1-atq_m)\bigr\|_2
\]
is a norm of a vector depending affinely on the real variable $t$, hence convex by the triangle inequality; the same holds for $\sqrt\gamma$ by Lemma~\ref{lem:direct}. Composing with the map $x\mapsto x^{2\rho}$ gives the claim.
\end{proof}

The exponents in \eqref{eq:directtest} and \eqref{eq:Tint} are $\frac{m-1}{2}$ and $\frac{m-2}{2}$, both at least $1/2$ for $n\ge4$, so Lemma~\ref{lem:chord} applies to each majorant. The weights are nonnegative rationals, computed exactly and enclosed outwards, so $g(\ell)$, $g(h)$ may be replaced by the certified upper bounds \eqref{eq:explicitbounds}. Endpoint bounds for $\beta$ may be capped by one using \eqref{eq:betachord}, and the $\gamma$ majorant is used on a subinterval only when both its endpoint bounds are at most one, so convexity gives $\gamma\le1$ throughout. For the denominator in \eqref{eq:Tint}, we have $|D(t)|\ge1-h\sqrt{Ar}$ on $[\ell,h]$ by \eqref{eq:explicitbounds}; when this is positive, the chord integral of the numerator is divided by this quantity, and otherwise only the second entry of the minimum is used.

To bound an integral over $[0,1]$, the chord rule is applied on each interval of a finite rational mesh $0=t_0<t_1<\cdots<t_N=1$ and the results are added. Two meshes are used. The first consists of the points $j/q$ and $j/(16q)$, $0\le j\le q$: it is uniform on $[0,1]$, with extra points near $t=0$, where the powers $\beta(t)^{(m-1)/2}$ decay quickly when $m$ is large. The second, taken from \cite[Section~7]{Zhang26}, is refined toward both endpoints: each gap of $0,2^{-15},2^{-14},\ldots,2^{-1},1$ is divided into $\max(1,\lceil q/12\rceil)$ equal parts, and the reflection $t\mapsto1-t$ of every point is added. The three-parameter checker uses the first mesh; in the two-parameter certificate, each degree block records which mesh it uses.
 
\subsection{Degree blocks}\label{subsec:blocks}
 
For degrees $27\le n\le2047$, we use two further simplifications.

First, $r$ is not subdivided: the regions are the prisms $\{(A,r,v)\in\mathcal D_n: A\in I_A,\ r+v\in I_s\}$ over boxes $I_A\times I_s$ in the coordinates $(A,s)$, so that $r$ ranges over the entire interval allowed by Proposition~\ref{prop:restrictions} and $v=s-r$. Apart from $b$, which Remark~\ref{rem:boxeval} bounds below independently of $r$, the direct test \eqref{eq:directtest} involves $r$ only through the term $\frac mn K\sqrt r$, so we replace $\sqrt r$ by its largest admissible value. The centered test \eqref{eq:centeredtest} depends on $r$ nonlinearly and needs an extra justification. In this case, we bound every ingredient of $\mathcal E_{\mathrm{int}}$ independently of $r$ (for instance, $|D(t)|$ from below by \eqref{eq:explicitbounds} at the largest admissible $r$), except the factor $\sqrt r$ of the prefactor $n\sqrt{Ar}$ in \eqref{eq:Tint}. This gives $\mathcal E_{\mathrm{int}}\le H+C\sqrt r$ with constants $H,C\ge0$ on the box. Writing $z=\sqrt r\in[0,\sqrt s]$, the right-hand side of \eqref{eq:centeredtest} is then at most $\Psi(z):=Kz+(s-z^2)(H+Cz)$, whose derivative $\Psi'(z)=K+Cs-2Hz-3Cz^2$ decreases on $[0,\sqrt s]$. Hence, if
\begin{equation}\label{eq:cubicmono}
 K\ge2H\sqrt s+2Cs,
\end{equation}
then $\Psi'(\sqrt s)=K-2H\sqrt s-2Cs\ge0$, so $\Psi$ is nondecreasing on $[0,\sqrt s]$ and, by Lemma~\ref{lem:products}, $\Psi(\sqrt r)\le\Psi(\sqrt s)=K\sqrt s\le s^{n/2}<1$. This contradicts \eqref{eq:centeredtest}, so the region is excluded.

Second, one box may serve a block $n_-\le n\le n_+$ of consecutive degrees if every bound is uniform in $n$. A block starts from $\{n_-,\ldots,n_+\}\times[0,n_+-1]\times[0,1]$ in the coordinates $(n,A,\eta)$, which loses nothing since $A\le m\le n_+-1$ by Lemma~\ref{lem:critical}. For uniformity, powers of bases in $[0,1]$ take their smallest exponent over the block in upper bounds and their largest in lower bounds, and nonnegative prefactors take their maximum. Moreover, $c_m<5/3$, $c_{m,2}<3$, and $m/n$, $k$, $L$ and the factor $\sqrt A\,((n-A)/m)^{m/2}$ of $K$ increase with $n$. For the last one, the logarithmic derivative relative to $m$ is $\tfrac12g(x)$ at $x=(1-A)/m>-1$, where $g(x)=\log(1+x)-\frac{x}{1+x}\ge0$ since $g(0)=0$ and $g'(x)=x/(1+x)^2$. For fixed $n$, this factor increases up to $A=1$ and decreases afterwards. Over an $A$-interval, its upper bound therefore comes from the point nearest $1$, and its lower bound, needed in \eqref{eq:cubicmono}, comes from the endpoints, or is $0$ if the interval reaches $n_-$.

Proposition~\ref{prop:restrictions}.\ref{prop:restrictions2} reads $\bigl(ns(1+A)/m\bigr)^m\ge n^2$, whose base exceeds one, so the rule for powers above does not apply. Since $n/(n-1)$ decreases in $n$, the base is at most $\theta:=n_-s(1+A)/(n_--1)$ on the block. The left-hand side is therefore at most $\max\{1,\theta\}^{\,n_+-1}$ at every degree, and the region is excluded if this is below $n_-^2\le n^2$. In a fixed degree, \ref{prop:restrictions2} also raises the lower end of $s$ to $\frac mn n^{2/m}/(1+A)$, with $n^{2/m}$ enclosed by integer bisection of its $m$th power.

\subsection{Results}\label{subsec:results}
 
\begin{proposition}
\label{prop:finite}
Theorem~\ref{thm:main} holds for every degree $4\le n\le2047$.
\end{proposition}
 
\begin{proof}[Computer-assisted verification]
For $4\le n\le26$, the three-dimensional certificate supplies one binary tree per degree, with $150\,845$ leaves in total, distributed as in Table~\ref{tab:joint}. For $27\le n\le2047$, the certificate of Subsection~\ref{subsec:blocks} supplies $27\,429$ leaves over the degree blocks of Table~\ref{tab:blocks}, which are checked to be consecutive, without gap or overlap, and to have union exactly $\{27,\ldots,2047\}$. The replay verifies every leaf of both certificates and finds none unresolved, so Lemma~\ref{lem:certificate} applies in every degree $4\le n\le2047$.
\end{proof}
 
\vspace{-3pt}
\begin{table}[htbp]
\centering
\caption{Three-parameter certificate: complete coverage in each degree.}
\vspace{-4pt}
\label{tab:joint}
\begin{tabular}{rr@{\qquad\quad}rr@{\qquad\quad}rr}
\toprule
Degree & Accepted regions & Degree & Accepted regions & Degree & Accepted regions\\
\midrule
4  & 12\,604 & 12 & 4\,441  & 20 & 3\,366\\
5  & 7\,435  & 13 & 3\,437  & 21 & 3\,041\\
6  & 6\,902  & 14 & 17\,030 & 22 & 2\,812\\
7  & 14\,288 & 15 & 10\,536 & 23 & 2\,562\\
8  & 16\,299 & 16 & 5\,215  & 24 & 2\,365\\
9  & 10\,095 & 17 & 3\,992  & 25 & 2\,160\\
10 & 7\,935  & 18 & 3\,381  & 26 & 1\,898\\
11 & 6\,299  & 19 & 2\,752  &    & \\
\midrule
\multicolumn{5}{r}{Total accepted regions} & 150\,845\\
\bottomrule
\end{tabular}
\end{table}
\vspace{-6pt}
 
\begin{table}[htbp]
\centering
\caption{Two-parameter certificate, allowing integer degree blocks.}
\vspace{-4pt}
\label{tab:blocks}
\begin{tabular}{rr@{\quad\quad}rr@{\quad\quad}rr}
\toprule
Degree & Accepted regions & Degree & Accepted regions & Degree & Accepted regions\\
\midrule
27     & 1\,885 & 72--107  & 1\,247 & 303--605   & 1\,536\\
28--31 & 3\,082 & 108--161 & 1\,267 & 606--1211  & 2\,202\\
32--47 & 4\,938 & 162--242 & 1\,814 & 1212--2047 & 2\,552\\
48--71 & 5\,461 & 243--302 & 1\,445 &            & \\
\midrule
\multicolumn{5}{r}{Total accepted regions} & 27\,429\\
\bottomrule
\end{tabular}
\end{table}

\subsection{Proof of the main theorem}
We are now finally ready to prove Theorem~\ref{thm:main}. 
 
\begin{proof}[Proof of Theorem~\ref{thm:main}]
After the degenerate cases, the normalization \eqref{eq:normal} reduces failure of the theorem at a prescribed zero to \eqref{eq:counter}. Theorem~\ref{thm:large} excludes this for $n\ge2048$ and Proposition~\ref{prop:finite} for $4\le n\le2047$, while the cases $n=2$ and $n=3$ were already proved in \cite{KPPSS}.
\end{proof}

\subsection*{Computational assistance}
OpenAI Codex was used in exploring the estimates, developing and running the certificate implementation, and proof-reading this manuscript. The accompanying material records the exact inputs and replay results of the verification.

\newpage
\appendix
\section{Implementation and reproduction}\label{app:package}
This appendix connects the mathematical exclusions with their names in the verifier, so that a recorded terminal test can be traced back to its proof, and specifies the exact finite inputs and the commands needed to replay their verification. The numerical tables in the paper summarize those inputs; the file digests below identify them.

\subsection{Implementation correspondence}\label{app:implementation}
The implementation is the package \path{borcea}, driven by \path{verify_all.py}: \path{interval.py} (directed integer arithmetic), \path{quadrature.py} (chord rule and meshes), \path{restrictions.py} (Proposition~\ref{prop:restrictions} and Lemma~\ref{lem:products}), \path{endpoint.py} (Proposition~\ref{prop:dual} and Corollary~\ref{cor:shifted}), \path{joint.py} (degrees $4$--$26$), \path{blocks.py} (degrees $27$--$2047$) and \path{certificate.py} (replay and coverage). The following table identifies the mathematical justification of the terminal exclusion families. A label records which check excluded a box first, so the same inequality may appear under different labels; every label is justified by the table below.

{
\renewcommand{\arraystretch}{1.33}
\begin{longtable}{>{\raggedright\arraybackslash}p{.355\textwidth}>{\raggedright\arraybackslash}p{.585\textwidth}}
\toprule
Exclusion family & Mathematical justification\\
\midrule
\endfirsthead
\toprule Exclusion family & Mathematical justification\\\midrule
\endhead
\bottomrule\endfoot
\texttt{domain} &
 $A\le m$ (Lemma~\ref{lem:critical}), which defines the search space.\\[4pt]
\texttt{derivative-product} &
 Proposition~\ref{prop:restrictions}.\ref{prop:restrictions2}, made uniform over degree blocks as in Subsection~\ref{subsec:blocks}. In a single degree, it also gives the lower bound $s\ge\frac mn n^{2/m}/(1+A)$.\\[4pt]
\texttt{restriction-1}; \texttt{restriction-3}; \texttt{restriction-5} &
 Proposition~\ref{prop:restrictions}.\ref{prop:restrictions1}, \ref{prop:restrictions3} and \ref{prop:restrictions5}, evaluated on boxes as in Remark~\ref{rem:boxeval}. The box is empty after $A$ is contracted by \ref{prop:restrictions3}, after $r$ is contracted to the admissible interval of Remark~\ref{rem:boxeval}, or after $v=s-r$ is contracted with it; or the lower bound for $b$ exceeds $2\sqrt{Ar}$.\\[4pt]
\texttt{negative-endpoint}; \texttt{negative-interior}; \texttt{negative-D} &
 $\beta(t)\ge0$ and $|D(1)|^2\ge0$, against the upper bounds \eqref{eq:betaD} and \eqref{eq:explicitbounds}.\\[4pt]
\texttt{direct} &
 Proposition~\ref{prop:direct}, with the integration rules of Subsection~\ref{subsec:quadrature}.\\[4pt]
\texttt{centered-coefficient}; \texttt{centered-coefficient-boundary} &
 Lemma~\ref{lem:coefficients} and Corollary~\ref{cor:Tcoeff}, with the criteria of Proposition~\ref{prop:centered}.\\[4pt]
\texttt{centered-integral}; \texttt{centered-integral-boundary} &
 Lemma~\ref{lem:centered} and Corollary~\ref{cor:Tint}, with the criteria of Proposition~\ref{prop:centered} and the integration rules of Subsection~\ref{subsec:quadrature}.\\[4pt]
\texttt{centered-monotone} &
 Equation~\eqref{eq:cubicmono} (two-parameter part only).\\[4pt]
\texttt{shifted:}$T$ & Corollary~\ref{cor:shifted}.\\[4pt]
\texttt{dual:}$T$ &
 Proposition~\ref{prop:dual}, with $\mu$ enclosed as in Remark~\ref{rem:boxeval}, using Proposition~\ref{prop:restrictions}.\ref{prop:restrictions4} (degrees $4$--$13$).\\[4pt]
\texttt{quartic-local} &
 Proposition~\ref{prop:quarticlocal}, recognized by \eqref{eq:recognize_local} through $2401v\le s$ (degree four only).
\end{longtable}
}

\subsection{Repository and exact inputs}\label{app:reproduce}
The public repository and its immutable release contain the verifier package with both certificates, the machine-readable report, and a checksum manifest:
\begin{center}
\begin{tabular}{ll}
Repository & \RepositoryURL\\
Immutable release & \ReleaseURL\\
Permanent archive & \ArchiveDOI
\end{tabular}
\end{center}
The two input files used for the reported computation have the following SHA-256 digests. Each two-line display is one concatenated 64-character hexadecimal string:
\begin{align*}
 \texttt{certificates/joint.json}:~
 &\texttt{8fb676db95c9e8d513d3fe2bc3ecc165}\\[-2pt]
 &\texttt{444cee13e38c876abd834f169ab15c41}.
\end{align*}
\vspace{-16pt}
\begin{align*}
 \texttt{certificates/blocks.json}:~
 &\texttt{69fe29be74825b9dba78e08c637e3ff5}\\[-2pt]
 &\texttt{7603cc5224f54dec17453929e59a76a8}.
\end{align*}
The manifest identifies the code and every accompanying data file. Digests establish file identity; they do not establish the inequalities or replace execution of the checker.
 
\subsection{Replaying the proof package}
Python 3.10 or later and the standard library suffice. From the root directory of the accompanying archive, run
\begin{lstlisting}
python verify_all.py --workers 6
\end{lstlisting}
This command runs the arithmetic, quadrature, coefficient-identity, quartic-constant, restriction-soundness, rejection and obstruction tests, replays both finite partitions, and verifies that they cover every degree from $4$ to $2047$ exactly once. The expected result is \texttt{PASS}, degrees $[4,2047]$, with $150\,845$ regions in the joint certificate and $27\,429$ in the block certificate. Each certificate is complete on its own stated range; no incomplete lower-degree certificate is needed. The worker count changes execution scheduling, not the inequalities or directed arithmetic.
 
The self-tests can also be run on their own, and the certificates can be named explicitly:
\begin{lstlisting}
python tests/selftest.py
python verify_all.py --joint certificates/joint.json --blocks certificates/blocks.json --workers 6
\end{lstlisting}
The report is \path{reports/report.json}, with one result for every degree from $4$ to $26$ and for every degree block. For an integer degree block, the reported number of regions refers to the whole block, not an independently generated partition for each integer in that block.
 
The certificates were generated by \path{tools/refine.py}, which the verifier neither imports nor needs. Its conclusion depends on reconstructing complete coverage and checking every leaf, rather than on a search budget or a previously recorded success flag.

\end{document}